\documentclass[12pt]{amsart}
\usepackage{amssymb,amscd,amsthm,amsmath,color}
\usepackage{fullpage}
\usepackage{epsfig}
\usepackage{graphicx}
\usepackage{xypic}
\usepackage{url}
\usepackage{color}
\usepackage{hyperref}
\usepackage{enumerate}
\usepackage{cjhebrew}

\numberwithin{equation}{section}

\newcommand{\OO}{\mathcal{O}}

\newcommand{\FF}{\mathcal{F}}

\newcommand{\ev}{\operatorname{ev}}

\newcommand{\rk}{\operatorname{rk}}
\newcommand{\cok}{\operatorname{cok}}

\newcommand{\Eff}{\overline{\operatorname{Eff}}}
\newcommand{\Nef}{\operatorname{Nef}}

\newcommand{\Mor}{\operatorname{Mor}}

\newcommand{\SP}{\operatorname{SP}}

\newcommand{\cU}{\mathcal{U}}

\newcommand{\Hilb}{\operatorname{Hilb}}
\newcommand{\Sec}{\operatorname{Sec}}

\newcommand{\Jac}{\operatorname{Jac}}

\newtheorem{theorem}{Theorem}[section]
\newtheorem{lemma}[theorem]{Lemma}
\newtheorem{proposition}[theorem]{Proposition}
\newtheorem{corollary}[theorem]{Corollary}

\theoremstyle{definition}
\newtheorem{definition}[theorem]{Definition}

\newtheorem{remark}[theorem]{Remark}

\newtheorem{example}[theorem]{Example}
\newtheorem{construction}[theorem]{Construction}

\xyoption{all}

\begin{document}

\title{Non-free curves on Fano varieties}
\author{Brian Lehmann}
\address{Department of Mathematics \\
Boston College  \\
Chestnut Hill, MA \, \, 02467}
\email{lehmannb@bc.edu}

\author{Eric Riedl}
\address{Department of Mathematics \\
University of Notre Dame  \\
255 Hurley Hall \\
Notre Dame, IN 46556}
\email{eriedl@nd.edu}

\author{Sho Tanimoto}
\address{Graduate School of Mathematics, Nagoya University, Furocho Chikusa-ku, Nagoya, 464-8602, Japan}
\email{sho.tanimoto@math.nagoya-u.ac.jp}


\begin{abstract}
Let $X$ be a smooth Fano variety over $\mathbb{C}$ and let $B$ be a smooth projective curve over $\mathbb{C}$. Geometric Manin's Conjecture predicts the structure of the irreducible components $M \subset \mathrm{Mor}(B, X)$ parametrizing curves which are non-free and have large anticanonical degree.  Following ideas of \cite{LRT22},  we prove the first prediction of Geometric Manin's Conjecture describing such irreducible components.  As an application, we prove that there is a proper closed subset $V \subset X$ such that all non-dominant components of $\mathrm{Mor}(B, X)$ parametrize curves in $V$, verifying an expectation put forward by Victor Batyrev. We also demonstrate two important ways that studying $\mathrm{Mor}(B,X)$ differs from studying the space of sections of a Fano fibration $\mathcal{X} \to B$.

\end{abstract}

\maketitle


\section{Introduction}

Let $X$ be a smooth complex Fano variety and let $B$ be a smooth complex projective curve.  We let $\Mor(B,X)$ denote the scheme parametrizing morphisms from $B$ to $X$ as in \cite[Chapter I 1.9 Definition and 1.10 Theorem]{Kollar}.
Recall that a morphism $s: B \to X$ is called a free curve if $s^{*}T_{X}$ is globally generated and $H^{1}(B,s^{*}T_{X}) = 0$.  Our goal is to classify the irreducible components of the morphism scheme $\Mor(B,X)$ which parametrize only non-free curves of large anticanonical degree.  Ever since Mori's inspirational work constructing rational curves on a Fano variety $X$, the moduli spaces $\Mor(B,X)$ have been a subject of intense investigation.  An important motivation behind our work comes from arithmetic geometry: by Batyrev's heuristics (\cite{Bat88}), one can deduce Manin's conjecture over global function fields from certain properties of components of the morphism scheme $\Mor(B,X)$.

\cite{LRT22} studies an analogous problem in a more general setting.  A Fano fibration $\pi: \mathcal{X} \to B$ is a morphism of smooth complex projective varieties such that the generic fiber of $\pi$ is a geometrically integral Fano variety.  We will denote by $\Sec(\mathcal{X}/B)$ the Hilbert scheme of sections of $\pi$.  Note that $\Mor(B,X)$ is isomorphic to the space of sections of the trivial fibration $\Sec(X \times B / B)$; we distinguish the two settings by referring to the study of $\Mor(B,X)$ as the ``absolute setting'' and the study of $\Sec(\mathcal{X}/B)$ as the ``relative setting''.

\cite{LRT22} showed that the irreducible components $M \subset \Sec(\mathcal{X}/B)$ which parametrize non-free sections can be classified using the Fujita invariant.  More precisely, such components come from $B$-morphisms $f: \mathcal{Y} \to \mathcal{X}$ which increase the Fujita invariant along the generic fiber.  Passing to the absolute setting, we can apply the results of \cite{LRT22} to the trivial fibration $\pi: X \times B \to B$.  However, it is natural to wonder whether the conclusions of the theorem also hold true in the absolute setting: can we account for non-free curves using morphisms of the form $f: Y \to X$ instead of morphisms of the form $f: \mathcal{Y} \to X \times B$? The question is subtle, and not all of the results of \cite{LRT22} hold in this context.

The goal of the present paper is to clarify which results from the relative setting carry over to the absolute setting. We show that the qualitative results about non-free curves mostly carry over and that the boundedness results carry over for non-dominant families of curves.  To prove these results we need to modify the arguments of \cite{LRT22} while keeping the same overall structure.  There are several minor points of distinction between the absolute and relative settings which we highlight throughout the paper.  On the other hand, we give counterexamples to show where key parts of the argument of \cite{LRT22} simply do not hold in the absolute setting. We would like to emphasize that whenever the two settings diverge, the relative setting seems to be more natural.

\subsection{Geometric Manin's Conjecture}

Geometric Manin's Conjecture is a set of principles that unifies predictions in the arithmetic setting (such as the function field version of Manin's Conjecture) and predictions in the geometric setting (such as the Cohen-Jones-Segal conjecture).  The key invariant in Geometric Manin's Conjecture is the Fujita invariant.

\begin{definition}
\label{defi:a-invariant}
Let $X$ be a smooth projective variety over a field of characteristic $0$ and let $L$ be a big and nef $\mathbb{Q}$-Cartier divisor on $X$.  The Fujita invariant of $(X,L)$ is
\begin{equation*}
a(X,L) = \min \{ t\in \mathbb{R} \mid  K_X + tL \textrm{ is pseudo-effective }\}.
\end{equation*}
If $L$ is nef but not big, we formally set $a(X,L) = \infty$.  If $X$ is singular, choose a resolution of singularities $\phi: X' \to X$ and define $a(X,L)$ to be $a(X',\phi^{*}L)$.  (The choice of resolution does not affect the value by \cite[Proposition 2.7]{HTT15}.)
\end{definition}

\subsection{Main results}

The first part of Geometric Manin's Conjecture predicts that all non-free curves come from morphisms $f: Y \to X$ which increase the Fujita invariant, and our results verify this prediction over the field $\mathbb C$ of complex numbers. It is the analogue in the absolute setting of \cite[Theorem 1.3]{LRT22}. Before stating our main theorem, let us introduce one piece of terminology we use frequently: let $Y \to T$ be a morphism between quasi-projective varieties such that a general fiber is irreducible. Let $Z \to T$ be any morphism between algebraic varieties such that the image meets with the Zariski open locus parametrizing irreducible fibers.
Then the fiber product $Z\times_T Y$ admits a unique irreducible component dominating $B$ which we call the ``main component.'' With this terminology, here is our main theorem:

\begin{theorem} \label{intro:maintheorem1}
Let $X$ be a smooth projective Fano variety defined over $\mathbb C$ and let $B$ be a complex smooth projective curve.  There are constants $\xi = \xi(\dim(X),g(B))$ and $\Gamma = \Gamma(\dim(X),g(B))$ such that the following holds.  Suppose that $M \subset \Mor(B,X)$ is an irreducible component parametrizing non-free maps $s: B \to X$ of anticanonical degree $\geq \xi$.  Let $\mathcal{U}^{\nu}$ be the normalization of the universal family over $M$ and let $ev: \mathcal{U}^{\nu} \to X$ be the evaluation map.  Then either:
\begin{enumerate}
\item $ev$ is not dominant.  Then the subvariety $Y$ swept out by the curves parametrized by $M$ satisfies $a(Y,-K_{X}) \geq a(X,-K_{X})$.
\item $ev$ is dominant and the general map parametrized by $M$ is birational onto its image. 
Let $\overline{\mathcal U}^\nu$ be a normal projective compactification of $\mathcal{U}^{\nu}$ with a morphism $\overline{ev} : \overline{\mathcal U}^\nu \to X$ extending $ev$. Then the finite part $f: Y \to X$ of the Stein factorization of $\overline{ev}$ satisfies
\begin{equation*}
a(Y,-f^{*}K_{X}) = a(X,-K_{X}).
\end{equation*}
Furthermore, there is a rational map $\phi: Y \dashrightarrow T$ such that the following properties hold.  Let $F$ denote a smooth resolution of the closure of a general fiber of $\phi$.  Then
\begin{enumerate}
\item $a(F,-f^{*}K_{X}|_{F}) = a(X,-K_{X})$.
\item The Iitaka dimension of $K_{F} - a(X,-K_{X})K_{X}|_{F}$ is equal to $0$.
\item 
Let $s: B \to X$ denote a general map parametrized by $M$, let $s': B \to Y$ denote the corresponding map to the Stein factorization and let $W$ denote the main component of $B \times_{T} Y$.  Then the image in $M$ of the parameter space of deformations of the map $(\mathrm{id},s'): B \to W$ has codimension at most $\Gamma$ in $M$.
\end{enumerate}
\item $ev$ is dominant and the general map parametrized by $M$ is not birational to its image.  In this case the image of the general map is a rational curve of anticanonical degree $2$.  Thus $ev$ factors rationally through a generically finite map $g: \mathcal{V} \to X$ where $\mathcal{V}$ is a projective model of a universal family $\mathcal U \to N$ of rational curves of anticanonical degree $2$ on $X$ where $N$ is an irreducible open locus of the Hilbert scheme of $X$.  
In particular $a(\mathcal{V},-g^{*}K_{X}) = a(X,-K_{X})$. 
\end{enumerate}
\end{theorem}

\begin{remark}
In the analogous result \cite[Theorem 1.3]{LRT22} condition (2).(b) is slightly different.  This highlights one distinction between the absolute and relative settings -- the Iitaka dimension does not behave the same.  See Theorem \ref{theo:absadjrigid} and Example \ref{exam:notadjointrigid} for details.
\end{remark}

This theorem is significant for two reasons.  First, it explicitly identifies the ``accumulating varieties'' which have more curves than expected.  Since the Fujita invariant is easier to work with than families of curves, we obtain a practical method for understanding non-free curves.  Second, by connecting the geometry of curves to the Fujita invariant we gain access to powerful methods from the Minimal Model Program.  In particular, \cite{Birkar21} yields a boundedness statement for the varieties considered in Theorem \ref{intro:maintheorem1}.

Next we prove the following theorem using the boundedness results of accumulating maps proved in \cite[Section 8]{LRT22}:

\begin{theorem} \label{intro:maintheorem2}
Let $X$ be a smooth projective Fano variety defined over $\mathbb C$ and let $B$ be a complex smooth projective curve.  There is a proper closed subset $V \subsetneq X$ such that if $M \subset \Mor(B,X)$ is an irreducible component parametrizing a non-dominant family of curves then every curve parametrized by $M$ is contained in $V$.
\end{theorem}

This theorem generalizes earlier results in \cite{LT19} and \cite{LT21b}. The proof essentially follows from boundedness results in \cite{LRT22}. However, we need slight modifications of the results in \cite{LRT22} in order to adjust to the absolute setting.  We carry out this project in Section~\ref{sec:boundedness}.

\begin{remark}
Boundedness theorems lead to an important distinction between the relative and absolute cases.  \cite[Theorem 8.8]{LRT22} shows that all non-free sections of a Fano fibration are accounted for by a bounded family of twists of finitely many morphisms $f: \mathcal{Y} \to \mathcal{X}$.  

In the absolute setting, one might guess that it is possible to ``ignore the twists'': is there a finite family of maps $f: Y \to X$ which account for all non-free curves?  In Example \ref{exam:needtwists} we demonstrate that this is not always possible: even in the absolute setting one must allow twists over $K(B)$.
\end{remark}

The following result is one of the key steps in the proof of Theorem \ref{intro:maintheorem1}.  We present it here as it gives new insight into the geometric significance of the Fujita invariant.  It is an analogue of \cite[Theorem 1.12]{LRT22}.

\begin{theorem} \label{theo:introainvandsections}
Let $X$ be a smooth projective Fano variety defined over $\mathbb C$ and let $B$ be a complex smooth projective curve.  Fix a positive rational number $a$ and a non-negative integer $T$.  There is some constant $\xi = \xi(\dim(X), g(B), a, T)$ with the following property.

Suppose that $Y$ is a smooth projective variety equipped with a morphism $f: Y \to X$ that is generically finite onto its image.   Suppose that $N$ is an irreducible component of $\Mor(B,Y)$ parametrizing a dominant family of curves $C$ on $Y$ which satisfy $-f^{*}K_{X} \cdot C \geq \xi$.  Finally, suppose that
\begin{equation*}
\dim(N) \geq a \cdot \dim(M) - T.
\end{equation*}
Then
\begin{equation*}
a(Y,-f^{*}K_{X}) \geq a.
\end{equation*}
\end{theorem}

\

\noindent
{\bf Acknowledgements:}
Brian Lehmann was supported by Simons Foundation grant Award Number 851129.
Eric Riedl was supported by NSF CAREER grant DMS-1945944.  Sho Tanimoto was partially supported by JST FOREST program Grant number JPMJFR212Z, by JSPS Bilateral Joint Research Projects Grant number JPJSBP120219935, by JSPS KAKENHI Grand-in-Aid (B) 23H01067, and by JSPS KAKENHI Early-Career Scientists Grant number 19K14512.

\section{Background} \label{sect:background}

Throughout we work with schemes whose irreducible components are finite type over a field $k$ of characteristic $0$.  A variety is a separated scheme of finite type over $k$ that is reduced and irreducible.  Given a coherent sheaf $\mathcal{F}$ on a variety $V$, we denote by $\mathcal{F}_{tors}$ its torsion subsheaf and by $(\mathcal{F})_{tf}$ the quotient of $\mathcal{F}$ by its torsion subsheaf.

When $X$ is a projective variety, we will let $N^{1}(X)_{\mathbb{R}}$ denote the space of $\mathbb{R}$-Cartier divisors up to numerical equivalence.  In this finite-dimensional vector space we have the pseudo-effective cone $\Eff^{1}(X)$ and the nef cone $\Nef^{1}(X)$.  Dually, we will let $N_{1}(X)_{\mathbb{R}}$ denote the space of $\mathbb{R}$-curves up to numerical equivalence.  Given a curve $C$, we will denote its numerical class by $[C]$.  Inside $N_{1}(X)_{\mathbb{R}}$ we have  the pseudo-effective cone $\Eff_{1}(X)$ and the nef cone $\Nef_{1}(X)$.

\subsection{Vector bundles on curves}
Throughout this paper, $B$ denotes a smooth projective curve defined over $\mathbb C$.
In this section we let $\mathcal{E}$ be a vector bundle of rank $r$ on $B$. We recall several definitions on semistability of $\mathcal E$:

\begin{definition}
Let $B$ be a complex smooth projective curve and let $\mathcal{E}$ be a vector bundle of rank $r$ on $B$.  
We define the slope of $\mathcal{E}$ on $B$ as 
\[
\mu(\mathcal{E}) = \frac{\deg \, \mathcal E}{r}.
\]
We say $\mathcal E$ is unstable if there is a non-zero subsheaf $\mathcal F \subset \mathcal E$ such that
\[
\mu (\mathcal F) > \mu (\mathcal E).
\]
When it is not unstable, we say $\mathcal E$ is semistable.

The Harder-Narasimhan filtration of $\mathcal E$ is a sequence of saturated subsheaves
\[
0 = \mathcal F_0 \subsetneq \mathcal F_1 \subsetneq \cdots \subsetneq \mathcal F_k = \mathcal E,
\]
such that for each $i$, $\mathcal F_{i + 1}/\mathcal F_i$ is semistable and we have
\[
\mu(\mathcal F_{i + 1}/\mathcal F_i) > \mu (\mathcal F_{i + 2}/\mathcal F_{i+1}).
\]
We denote $\mu(\mathcal F_{1})$ and $\mu(\mathcal F_k/\mathcal F_{k-1})$ by $\mu^{\max}(\mathcal E)$ and $\mu^{\min}(\mathcal E)$ respectively.
\end{definition}

\begin{definition}
We say that a coherent sheaf $\mathcal{E}$ on a smooth projective curve $B$ is generically globally generated if the evaluation map
\begin{equation*}
H^{0}(B,\mathcal{E}) \otimes \mathcal{O}_{B} \to \mathcal{E}
\end{equation*}
is surjective at the generic point of $B$.
\end{definition}

We will need the following results concerning the positivity of generically globally generated bundles on curves.

\begin{lemma}[{\cite[Lemma 2.5]{LRT22}}] \label{lemm:genericallygloballygenerated}
Let $B$ be a smooth projective curve.  Suppose that $\mathcal{E}$ is a generically globally generated vector bundle on $B$.  Then every successive quotient $\mathcal{F}_{i}/\mathcal{F}_{i-1}$ in the Harder-Narasimhan filtration of $\mathcal{E}$ satisfies $\mu(\mathcal{F}_{i}/\mathcal{F}_{i-1}) \geq 0$.
\end{lemma}

\begin{lemma}[{\cite[Lemma 2.8]{LRT22}}] \label{lemm:ggh1bound}
Let $B$ be a smooth projective curve of genus $g$.  Suppose that $\mathcal{E}$ is a generically globally generated vector bundle on $B$.  Then
\begin{enumerate}
\item $h^{0}(B,\mathcal{E}) \leq \deg(\mathcal{E}) + \rk(\mathcal{E})$.
\item $h^{1}(B,\mathcal{E}) \leq g(B) \rk(\mathcal{E})$.
\end{enumerate}
\end{lemma}

\subsection{Morphism spaces}  \label{sect:spaceofmorphisms}

Suppose that $Z$ is a smooth projective variety and $B$ is a smooth projective curve.  If $M \subset \Mor(B,Z)$ is an irreducible component, we have
\begin{equation*}
-K_{Z} \cdot s_{*}B + \dim(Z)(1-g(B)) \leq \dim(M) \leq h^{0}(B,s^{*}T_{Z})
\end{equation*}
where $s: B \to Z$ is any curve parametrized by $M$. (See \cite[Chapter 2, 1.7 Theorem]{Kollar} for this claim.) More generally, suppose we fix closed points $p_{1},\ldots,p_{m} \in B$ and closed points $q_{1},\ldots,q_{m} \in Z$ and consider the sublocus $\Mor(B,Z;p_{i} \mapsto q_{i})$ of morphisms $s$ such that $s(p_{i}) = q_{i}$ for every $i$.  If $M \subset \Mor(B,Z;p_{i} \mapsto q_{i})$ is an irreducible component then
\begin{equation*}
-K_{Z} \cdot s_{*}B + \dim(Z)(1-g(B)) - m \dim(Z) \leq \dim(M) \leq h^{0}(B,s^{*}T_{Z}(-p_{1}-\ldots-p_{m}))
\end{equation*}
where $s: B \to Z$ is any morphism parametrized by $M$. (Again see \cite[Chapter 2, 1.7 Theorem]{Kollar} for this claim.)

For a morphism $s: B \to Z$ to a smooth projective variety $Z$, we denote by $N_{s}$ the normal sheaf of $s$, i.e.~the cokernel of $T_{B} \to s^{*}T_{Z}$.  If $s$ is a general member of a dominant family of morphisms to $Z$ then the normal sheaf $N_{s}$ will be generically globally generated.  %

\begin{proposition} \label{prop:domfamilyexpdim}
Let $Z$ be a smooth projective variety and let $B$ be a smooth projective curve.  Suppose that $M$ is an irreducible component of $\Mor(B,Z)$ parametrizing a dominant family of curves.  Letting $s$ denote a general map parametrized by $M$, we have
\begin{equation*}
-K_{Z} \cdot s_{*}B + \dim(Z)(1-g(B)) \leq \dim(M) \leq -K_{Z} \cdot s_{*}B + \dim(Z) + 2g(B).
\end{equation*}
\end{proposition}

\begin{proof} 
In this situation the normal sheaf $N_{s}$ is generically globally generated. Indeed, let $V$ be the tangent space of $M$ at $s$.  We can identify $V$ as a subspace of $H^0(B, s^*T_Z)$. Let $\pi : \mathcal U = M \times B \to M$ be the universal family over $M$ with the evaluation map $ev : \mathcal U \to Z$. Then the tangent space of $\mathcal U$ at a general point $p \in B$ above $s$ is given by $V \oplus T_{B, p}$.  Because $ev$ is dominant, the evaluation map induces a surjection $V \oplus T_{B, p} \to T_{Z, s(p)}$. In particular, this implies that $H^0(B, s^*T_Z) \oplus T_{B, p} \to T_{Z, s(p)}$ is surjective. This means that $H^0(B, s^*T_Z) \to N_{s, p} = s^*T_{Z, s(p)}/T_{B, p}$ is surjective. Since $H^0(B, s^*T_Z) \to N_{s, p}$ factors through $H^0(B, N_s) \to N_{s, p}$, we conclude that $H^0(B, N_s) \to N_{s, p}$ is surjective. Hence $N_s$ is generically globally generated. 
Now we have
\begin{align*}
h^{1}(B,s^{*}T_{Z}) & \leq h^{1}(B,T_{B}) + h^{1}(B,N_{s}) \\
& = h^{1}(B,T_{B}) + h^{1}(B,(N_{s})_{tf}) \\
& \leq 3g(B) + (\dim(Z)-1)g(B)
\end{align*}
where in the last line we have used the elementary inequality $h^{1}(B,T_{B}) \leq 3g(B)$ and have applied Lemma \ref{lemm:ggh1bound} to $(N_{s})_{tf}$.  The desired statement follows.
\end{proof}

The deformation theory of a map $s: B \to Z$ is best behaved under a stronger assumption on the positivity of $s^{*}T_{Z}$.

\begin{definition}
Let $Z$ be a smooth projective variety and let $B$ be a smooth projective curve.  We say that a map $s: B \to Z$ is HN-free if $\mu^{min}(s^{*}T_{Z}) \geq 2g(B)$.
\end{definition}

The following result summarizes the key properties of HN-free maps.

\begin{lemma}[{\cite[Lemma 3.6]{LRT22}}]  \label{lemma:hnfreecurves}
Let $Z$ be a smooth projective variety and let $B$ be a smooth projective curve.  Suppose that $s: B \to Z$ is an HN-free map.  Then:
\begin{enumerate}
\item $H^{1}(B,s^{*}T_{Z}) = 0$ and for any closed point $p \in B$ we have $H^{1}(B,s^{*}T_{Z}(-p)) = 0$.
\item $s^{*}T_{Z}$ is globally generated.
\item Let $b$ be the smallest slope of a quotient of successive terms in the Harder-Narasimhan filtration for $s^{*}T_{Z}$.  Then deformations of $s : B \to Z$ can pass through at least $\lfloor b \rfloor - 2g(B) + 1$ general points of $Z$.
\end{enumerate}
\end{lemma}

\begin{proof}
In the notation of \cite[Lemma 3.6]{LRT22}, $\mathcal Z = Z \times B$ and $T_{\mathcal Z/B}|_C$ is $s^*T_Z$. With these translations, the statements follow from \cite[Lemma 3.6]{LRT22}.
\end{proof}

The previous lemma shows that HN-free maps go through many general points of $Z$. 
Conversely, maps through sufficiently many general points of the product $B \times Z$ must be HN-free.

\begin{proposition}[{\cite[Proposition 3.7]{LRT22}}] \label{prop:generalimplieshnfree}
Let $Z$ be a smooth projective variety and let $B$ be a smooth projective curve.  Let $M$ be a component of $\Mor(B,Z)$. Equivalently, we can think of $M$ as parametrizing a family of sections of the projection map $B \times Z \to B$.  Suppose that the sections parametrized by $M$ pass through $\geq 2g(B)+1$ general points of $B \times Z$.  Then the general curve parametrized by $M$ is HN-free.
\end{proposition}

\begin{example}
In the setting of Proposition \ref{prop:generalimplieshnfree} we really need to work with the product $B \times Z$ and not the variety $Z$.  For example, suppose that $Z$ is a $\mathbb{P}^{1}$-bundle over a high genus curve $T$.  Then sections of $Z \to T$ can contain arbitrarily many general points of $Z$ but will never be HN-free.  This issue will come up again in Example \ref{exam:notadjointrigid}.
\end{example}

We will also need to know the following avoidance property of HN-free maps.

\begin{lemma}[{\cite[Lemma 3.8]{LRT22}}] \label{lemm:HNavoidscodim2}
Let $Z$ be a smooth projective variety and let $B$ be a smooth projective curve.  Suppose that $s: B \to Z$ is an HN-free map.  Then for any codimension $2$ closed subset $W \subset Z$ there is a deformation of $s$ which is HN-free and whose image avoids $W$. 
\end{lemma}

\subsection{Fujita invariant}  

Recall from Definition~\ref{defi:a-invariant} that for a smooth projective variety $X$ defined over a field of characteristic $0$ and a big and nef $\mathbb{Q}$-Cartier divisor $L$ on $X$, we define the Fujita invariant by
\begin{equation*}
a(X,L) = \min \{ t\in \mathbb{R} \mid  K_X + tL \in \Eff^{1}(X) \}.
\end{equation*}

It follows from the seminal work \cite{BDPP13} that the Fujita invariant will be positive if and only if $X$ is geometrically uniruled.  We will use the Spectrum Conjecture for Fujita Invariants which was first established in slightly different settings by DiCerbo and by Han and Li using Birkar's solution of the BAB conjecture (\cite{birkar19}):

\begin{theorem}[{\cite[Theorem 1.2]{Dicerbo17}}, {\cite[Theorem 1.3]{HL20}}]
\label{theo:Dicerbo}
Fix a positive integer $n$ and fix $\epsilon > 0$.  Let $X$ vary over all smooth projective varieties of dimension $n$ defined over a field of characteristic $0$ and $L$ vary over all big and nef Cartier divisors on $X$, the set
\[
\{a(X, L) \, | \, a(X, L) \geq \epsilon\}
\]
is finite.
\end{theorem}

The following definition is frequently useful when working with the Fujita invariant.

\begin{definition}
Let $X$ be a smooth projective variety and let $L$ be a big and nef $\mathbb{Q}$-divisor on $X$.  A pair $(X,L)$ is adjoint rigid if $K_{X} + a(X,L)L$ has Iitaka dimension $0$.  When $X$ is singular and $L$ is a big and nef $\mathbb{Q}$-Cartier divisor on $X$, we say $(X,L)$ is adjoint rigid if $(X',\phi^{*}L)$ is adjoint rigid for a smooth resolution $\phi: X' \to X$. This definition does not depend on the choice of resolution.
\end{definition}

The Fujita invariant can be used to bound adjoint rigid subvarieties.  Indeed, let $X$ be a smooth projective variety defined over an algebraically closed field and $L$ be a big and nef $\mathbb Q$-divisor on $X$. Then \cite[Theorem 4.17]{LST18} shows that the subvarieties $Y \subset X$ such that $L|_Y$ is big, $a(Y, L) \geq a(X, L)$, and $(Y, L|_Y)$ is adjoint rigid, are parametrized by a bounded family.

\subsection{Slope stability for smooth projective varieties}

The notion of slope stability with respect to movable curve classes was developed by \cite{CP11} and subsequently by \cite{GKP14} and \cite{GKP16}.

\begin{definition}
Let $X$ be a smooth projective variety and let $\alpha \in \Nef_{1}(X)$ be a non-zero nef $1$-cycle on $X$.  For any non-zero torsion-free sheaf $\mathcal{E}$ on $X$, we consider the following invariant:
\begin{equation*}
\mu_{\alpha}(\mathcal{E}) = \frac{c_{1}(\mathcal{E}) \cdot \alpha}{\rk(\mathcal{E})},
\end{equation*}
which is called as the slope of $\mathcal E$ with respect to a class $\alpha$.
A torsion-free sheaf $\mathcal{E}$ is $\alpha$-semistable if we have $\mu_{\alpha}(\mathcal{F}) \leq \mu_{\alpha}(\mathcal{E})$ for every non-zero torsion-free subsheaf $\mathcal{F} \subset \mathcal{E}$.
\end{definition}

\begin{definition}
Let $X$ be a smooth projective variety and let $\alpha \in \Nef_{1}({X})$ be a nef $1$-cycle on $X$.  Let $\mathcal{E}$ be a torsion-free sheaf of rank $r > 0$.  Suppose that
\begin{equation*}
0 = \mathcal{F}_{0} \subset \mathcal{F}_{1} \subset \mathcal{F}_{2} \subset \ldots \subset \mathcal{F}_{k} = \mathcal{E}
\end{equation*}
is the $\alpha$-Harder-Narasimhan filtration of $\mathcal{E}$.
The slope panel $\SP_{\alpha}(\mathcal{E})$ is the $r$-tuple of rational numbers obtained by combining for every index $i$ the list of $\rk(\mathcal{F}_{i}/\mathcal{F}_{i-1})$ copies of $\mu_{\alpha}(\mathcal{F}_{i}/\mathcal{F}_{i-1})$ (arranged in non-increasing order).  We define $\mu^{max}_{\alpha}(\mathcal{E})$ as the maximal slope of any torsion-free subsheaf, i.e.,~$\mu^{max}_{\alpha}(\mathcal{E}) = \mu_{\alpha}(\mathcal{F}_{1})$.  We define $\mu^{min}_{\alpha}(\mathcal{E})$ as the minimal slope of any torsion-free quotient, i.e.,~$\mu^{min}_{\alpha}(\mathcal{E}) = \mu_{\alpha}(\mathcal{E}/\mathcal{F}_{k-1})$.

When we discuss slope panels of a curve, we will always assume that $\alpha$ is a degree $1$ line bundle and thus we will simply write $\mu^{max}(\mathcal{E})$, $\mu^{min}(\mathcal{E})$, $\SP(\mathcal{E})$.
\end{definition}

 Suppose that $X$ is a smooth projective variety.  Recall that a foliation $\mathcal{F}$ on $X$ is a coherent subsheaf of the tangent bundle $\mathcal{F} \subset T_{X}$ such that that is closed under the Lie bracket of $T_{X}$, i.e.~$[\mathcal{F}, \mathcal{F}] \subset \mathcal{F}$.

\begin{theorem}[{\cite[Proposition 1.3.32]{Pang}}] \label{theo:HNisfoliation}
Suppose that $X$ is a smooth projective variety and let $\alpha \in \Nef_{1}(X)$ be a non-zero nef curve class.  Let 
\begin{equation*}
0 = \mathcal{F}_{0} \subset \mathcal{F}_{1} \subset \ldots \subset \mathcal{F}_{k} = T_{X}.
\end{equation*}
be the $\alpha$-Harder-Narasimhan filtration of $T_{X}$.
Then every term $\mathcal{F}_{i}$ with $\mu_{\alpha}^{min}(\mathcal{F}_{i}) > 0$ defines a foliation on $X$.
\end{theorem}

Given a rational map $f: X \dashrightarrow Y$ from a smooth projective variety $X$ to a normal projective variety $Y$, the fibers of $f$ define leaves of a foliation where the map is smooth.  The foliation induced by $f$ is the corresponding saturated subsheaf of $T_{X}$.   
\subsection{Families of non-birational maps}

Let $Z$ be a smooth projective variety and let $B$ be a smooth projective curve.  Suppose that $M$ is an irreducible component of $\Mor(B,Z)$ such that the general map parametrized by $M$ is not birational onto its image.  In this situation, we will show that there is an intermediate curve $B'$ such that a general map $s: B \to Z$ parametrized by $M$ factors through a morphism $s': B' \to Z$ that is birational onto its image.

\begin{lemma} \label{lemm:constantimage}
Suppose $B$ is a smooth projective curve of genus $g$.  Suppose we have a family of degree $d$ maps from $B$ to smooth projective curves $\{C_{t}\}$, i.e.~a commuting diagram
\begin{equation*}
\xymatrix{
B \times T \ar[rr] \ar[dr] & & \mathcal{C} \ar[dl] \\
& T &
}
\end{equation*}
where $\mathcal{C} \to T$ is a smooth proper morphism of relative dimension $1$ and $T$ is connected.  Then the set $\{B \to C_{t}\}$ of maps from $B$ are all isomorphic to each other.  
\end{lemma}

\begin{proof}
By taking relative Jacobians, we obtain a family of maps of abelian varieties $\Jac(B) \times T \to \mathcal{J}$.  We first claim that the fibers of $\mathcal{J} \to T$ are all isomorphic to each other.  Consider the Stein factorization $\Jac(B)\times T \to \mathcal A \to \mathcal{J}$.  Then $\mathcal A_{t}$ is an abelian variety and $\mathcal A_{t} \to \mathcal{J}_{t}$ is finite of a fixed degree.

First note that the morphism $\Jac (B) \times T \to \mathcal A$ over $T$ is isotrivial.  Indeed, a contraction morphism from $\Jac(B)$ is determined by the numerical class of the pullback of an ample divisor on $\mathcal A_t$.  Let $L$ be a relatively ample divisor on $\mathcal A$.  
Then the numerical class of the pullback of $L|_{\mathcal A_t}$ does not depend on $t$ because the numerical class is invariant under flat deformations, proving the claim.

Next note that $\mathcal{J} \to T$ is isotrivial.  It suffices to prove that the kernel of $\Jac(B) \to \mathcal J_t$ does not depend on $t$. Note that the connected component of the kernel of $\Jac(B) \to \mathcal J_t$ is the kernel of $\Jac(B)\to \mathcal A_t$ which does not depend on $t$. Thus we need to show that the order of the kernel of $\mathcal A_t \to \mathcal J_t$ does not depend on $t$. However, this follows from the fact that the degree of $\mathcal A_{t} \to \mathcal{J}_{t}$ is constant.
The upshot is that the kernel of $\Jac(B) \to \mathcal J_t$ does not depend on $t$ and this implies that 
$\mathcal J \to T$ is isomorphic to $J \times T \to T$ as relative abelian varieties where $J$ is an abelian variety.

To prove that the family $\mathcal C \to T$ is isotrivial, it suffices to consider the relative principal polarization of $\mathcal J \to T$ by the Torelli theorem.  However, since $H^{2}(\mathcal{J}_{t},\mathbb{Z})$ is discrete and the monodromy is trivial, any deformation will fix the numerical class of the polarization, proving that $C_t$ are all isomorphic to each other. Moreover since $\Jac(B) \to \Jac(C_t)$ is identified with $\Jac(B) \to J$, the maps $B \to C_{t}$ are all isomorphic to each other. Thus our assertion follows.
\end{proof}

\begin{corollary} \label{coro:factoringresult}
Let $Z$ be a smooth projective variety and let $B$ be a smooth projective curve.  Suppose that $W$ is a locally closed subvariety of $\Mor(B,Z)$ such that the morphisms parametrized by $W$ are not birational onto their image.  Then there is a finite morphism $g: B \to B'$ to a smooth projective curve $B'$ and a locally closed subvariety $Y \subset \Mor(B',Z)$ such that the general morphism parametrized by $W$ is the composition of $g$ with a morphism parametrized by $Y$.
\end{corollary}

\begin{proof}
Apply Lemma \ref{lemm:constantimage} to the family of maps of smooth curves obtained by taking the normalizations of the images of the general maps parametrized by $W$. 
\end{proof}

\section{Grauert-Mulich} \label{sect:gm}

In the remaining of this paper, we work over $\mathbb C$ and we let $B$ denote a complex smooth projective curve.
Suppose that $Z$ is a smooth projective variety and $W \subset \Mor(B,Z)$ is a variety parametrizing a family of maps $s: B \to Z$ with universal family $\mathcal{U}_{W} \to W$.  In this section we will usually assume that the evaluation map $ev: \mathcal{U}_{W} \to Z$ is dominant and that a general fiber over $W$ is contained in the flat locus of $ev$.

Let $\mathcal{E}$ be a torsion-free sheaf on $Z$.  Under the assumptions above, the Grauert-Mulich theorem of \cite{LRT22} shows that for a general curve $s: B \to Z$ parametrized by $W$ the Harder-Narasimhan filtration of $s^{*}\mathcal{E}$ is ``approximately'' the pullback of the Harder-Narasimhan filtration of $\mathcal{E}$ with respect to $s_{*}B$.

\begin{definition}
Let $Y$ be a variety and $\mathcal{E}$ be a globally generated vector bundle on $Y$.  The syzygy bundle (or Lazarsfeld bundle) $M_{\mathcal{E}}$ is the kernel of the evaluation map $\OO_Y\otimes H^0(Y,\mathcal{E})\to \mathcal{E}$.
\end{definition}

\begin{theorem}[\cite{LRT22} Corollary 4.6] \label{theo:generalgm}

Let $Z$ be a smooth projective variety and let $\mathcal{E}$ be a torsion free sheaf on $Z$ of rank $r$.  Let $W$ be a variety equipped with a generically finite morphism $W \to \overline{\mathcal M}_{g,0}(Z)$ and let $p: \mathcal{U}_{W} \to W$ denote the universal family over $W$ with evaluation map $ev_{W}: \mathcal{U}_{W} \to Z$.  Assume that a general map parametrized by $W$ has smooth irreducible domain, that $ev_{W}$ is dominant, that the general fiber of the composition of the normalization map for $\mathcal{U}_{W}$ with $ev_{W}$ is connected, and that a general fiber of $p$ is contained in the locus where $ev_{W}$ is flat.

Let $C$ denote a general fiber of $\cU_{W} \to W$ equipped with the induced morphism $s: C \to Z$.  Let $t$ be the length of the torsion part of $N_{s}$, let $\mathcal{G}$ be the subsheaf of $(N_{s})_{tf}$ generated by global sections, and let $V$ be the tangent space to $W$ at $s$. Let $q$ be the dimension of the cokernel of the composition
\begin{equation*}
V \to T_{ \overline{\mathcal M}_{g,0}(Z),s} = H^{0}(C,N_{s})  \to H^{0}(C,(N_{s})_{tf}).
\end{equation*}
Let $\Vert - \Vert$ denote the sup norm on $\mathbb{Q}^{\oplus r}$. Then we have
\begin{equation*}
\Vert \SP_{Z,[C]}(\mathcal{E}) - \SP_{C}(s^{*}\mathcal{E}) \Vert \leq \frac{1}{2} \left( (q+1)\mu^{max}(M_{\mathcal{G}}^{\vee}) + t \right)  \rk(\mathcal{E}).
\end{equation*}
\end{theorem}

To apply Theorem \ref{theo:generalgm} in practice one needs to be able to bound the quantities $q$, $t$, and $\mu^{max}(M_{\mathcal{G}}^{\vee})$ appearing in the statement.  In this section we will show how to bound these quantities using the genus of $B$ and the dimension of $Z$.  Using these bounds, we will obtain a version of the Grauert-Mulich Theorem for $\Mor(B,Z)$.

The following result of \cite{Butler94} (explained carefully in \cite[Theorem 4.8]{LRT22}) allows us to bound the slope of the syzygy bundle for the restricted tangent bundle. 

\begin{theorem}  \label{theo:butler}
Let $\mathcal{E}$ be a globally generated locally free sheaf on a curve $C$ of genus $g$ and let $M_{\mathcal{E}}$ be its syzygy bundle.
\begin{enumerate}
\item If $\mu^{min}(\mathcal{E}) \geq 2g$ then $\mu^{min}(M_{\mathcal{E}}) \geq -2$.
\item If $\mu^{min}(\mathcal{E}) < 2g$ then $\mu^{min}(M_{\mathcal{E}}) \geq -2g \rk(\mathcal{E}) - 2$.
\end{enumerate}
\end{theorem}

We next discuss the quantity $t$.  The key to bounding the ramification of a morphism $s: B \to Z$ is the following result of Arbarello and Cornalba.

\begin{theorem}[{\cite[Corollary 6.11]{AC81}}]
\label{thm-ArbarelloCornalba} 
Let $Z$ be a smooth projective variety defined over $\mathbb C$ and let $B$ be a complex smooth projective curve.  Suppose that $W \subset \Mor(B,Z)$ is a locally closed reduced subvariety such that the maps parametrized by $W$ dominate $Z$ and the general map parametrized by $W$ is birational onto its image.  Let $s$ be a general map parametrized by $W$. Then the image of the map $H^0(B,s^{*}T_{Z}) \to H^{0}(B,N_{s})$ has vanishing intersection with $(N_{s})_{tors}$.
\end{theorem}

Using this result we can bound the ramification of a map $s: B \to Z$ in terms of the genus of $B$ and the dimension of $Z$. 

\begin{proposition} \label{prop:torsionestimate}
Let $Z$ be a smooth projective variety defined over $\mathbb C$ and let $B$ be a complex smooth projective curve.  Suppose that $W$ is a locally closed subvariety of $\Mor(B,Z)$ such that the morphisms parametrized by $W$ dominate $Z$.  Let $s$ be a general element of $W$.
\begin{enumerate}
\item If the general map $s$ is birational onto its image then the length of the torsion of $N_{s}$ is at most $3g(B)$.
\item If the general map $s$ factors as the composition of a finite morphism of smooth curves $h: B \to B'$ followed by a morphism $s': B' \to Z$ that is birational onto its image, then the length of the torsion of $N_{s}$ is $3g(B')d + r$
where $d$ is the degree of $h$ and $r$ is the total ramification degree of $h$.  In particular, if $g(B') \geq 2$ then the length of the torsion of $N_{s}$ is at most $6g(B)$.
\end{enumerate} 
\end{proposition}

\begin{proof}
(1) We have the short exact sequence
\[ H^0(B,s^*T_Z) \to H^0(B,N_{s}) \to H^1(B,T_B) . \]
By Theorem \ref{thm-ArbarelloCornalba} we see the image of $H^0(B,s^*T_Z)$ in $H^0(B,N_{s})$ is disjoint from $H^0(B,(N_{s})_{tors})$.  Since the cokernel of $H^0(B,s^*T_Z) \to H^0(B,N_{s})$ injects into $H^1(B,T_B)$, it follows that $H^0(B,(N_{s})_{tors})$ has dimension at most $h^1(B,T_B) \leq 3g(B)$. 

(2) Consider the diagram
\begin{equation*}
\xymatrix{
0 \ar[r] & T_{B} \ar[r] \ar[d] & s^{*}T_{Z} \ar[r] \ar[d]_{=} & N_{s} \ar[r] \ar[d]_{\psi} & 0 \\
0 \ar[r] & h^{*}T_{B'} \ar[r] & s^{*}T_{Z} \ar[r] & h^{*}N_{s'} \ar[r] & 0
}
\end{equation*}  
Note that $(h^{*}N_{s'})_{tors}$ has length $d$ times the length of $(N_{s'})_{tors}$ which by (1) is at most $3g(B')$.  By the snake lemma, $\psi$ is surjective and its kernel is a torsion sheaf of length at most $r$.  Altogether this proves the first statement.  The final statement follows from the Riemann-Hurwitz formula $g(B) = dg(B') + \frac{r}{2} - d + 1$.   Indeed, we have
\begin{align*}
6g(B) = 6dg(B') + 3r - 6d + 6 & = 3dg(B') + r + (3dg(B')+2r-6d+6) \\
& \geq 3dg(B') + r.
\end{align*}
\end{proof}

In many situations one can show that a general map will not have any ramification at all.  The following result is an analogue of \cite[II.3.14 Theorem]{Kollar}.

\begin{proposition}
\label{prop-unramFrank2}
Let $Z$ be a smooth projective variety.  Suppose that $M$ is an irreducible component of $\Mor(B, Z)$ that is generically reduced such that the morphisms parametrized by $M$ dominate $Z$.  Let $s$ be a general element of $M$ and consider the Harder-Narasimhan filtration $0 = \mathcal{F}_{0} \subset \ldots \subset \mathcal{F}_{r} = s^{*}T_{X}$.  Suppose there is a term $\mathcal{F}$ in this filtration which has rank at least $2$ such that $\mu^{min}(\FF) > 2g(B)$.  Then the general $s$ parametrized by $M$ must be unramified. 
\end{proposition}

In particular, if $s^*T_Z$ contains a subsheaf $\FF$ of rank at least 2 such that $\mu^{min}(\FF) > 2g(B)$ then the general $s$ must be unramified.

\begin{proof}
We start with some reminders about deformations of ramified maps.  Suppose that $s: B \to Z$ is a map parametrized by $M$ which has a ramification point $p$ and let $z = s(p)$.  Let $M_{p,z}$ be the space of morphisms $s:B \to Z$ such that $s(p) = z$ and $M_{p,z}^{\text{ram}}$ be the space of morphisms $s: B \to Z$ such that $s(p) = z$ and $s$ is ramified at $p$.  Then the tangent space to $M_{p,z}$ at $s$ is $H^0(B,s^*T_Z(-p))$ and the tangent space to $M_{p,z}^{\text{ram}}$ at $s$ is $H^0(B,s^*T_Z(-2p))$.

Let $\pi_{1},\pi_{2}$ denote the two projection maps on $M \times B$.  Suppose for a contradiction that a general map parametrized by $M$ is ramified.  Consider the composition $ev^{*}\Omega_{Z} \to \Omega_{M \times B} \cong \pi_{1}^{*}\Omega_{M} \oplus \pi_{2}^{*}\Omega_{B} \to \pi_{2}^{*}\Omega_{B}$.
 We let $\mathcal{R} \subset M \times B$ be the closed subscheme whose ideal sheaf is the image of the corresponding morphism $\ev^{*}\Omega_{Z} \otimes \pi_{2}^{*}T_{B} \to \mathcal{O}_{M \times B}$.

We fix an irreducible component $\mathcal{R}_{0}$ of $\mathcal{R}$ equipped with the reduced structure.
We let $Y$ be the irreducible subvariety of $Z$ obtained by taking the closure of the image of $\mathcal{R}_{0}$.  If $(s,p)$ is a general point of $\mathcal{R}_{0}$, then it follows from generic smoothness for the restriction of $ev$ to the smooth locus of $\mathcal{R}_{0}$ that the image of $d_{\ev}: T_{(s,p)}M \times B \to T_{s(p)}Z$ contains $T_{s(p)}Y$.   However, since $s$ ramifies at $p$ the image of $T_{p}B$ under this map is zero.  Thus, the image of the corresponding map $H^0(B,s^*T_Z) = T_{s}M \to s^{*}T_Z|_p$ has  dimension at least $\dim Y$.  

Since $M$ is generically reduced we have $\dim M = h^0(B,s^*T_Z)$.  Thus for a general point $(s,p)$ in $\mathcal{R}_{0}$ we have \begin{align*}
h^0(B,s^* T_Z(-p)) & = h^0(B,s^*T_Z) - \dim \mathrm{im} \left(H^{0}(B,s^{*}T_{Z}) \to H^{0}(p,s^{*}T_{Z}|_{p}) \right) \\
& \leq  \dim M - \dim Y
\end{align*}
Since the minimal slope of a quotient of $\mathcal{F}$ is greater than $2g(B)$, $\mathcal{F}(-p)$ is globally generated.  Thus
\begin{align*}
h^0(B,s^* T_Z(-2p)) & = h^0(B,s^* T_Z(-p)) - \dim \mathrm{im} \left(H^{0}(B,s^{*}T_{Z}(-p)) \to H^{0}(p,s^{*}T_{Z}(-p)|_{p}) \right)  \\
& \leq (\dim M - \dim Y) -  \dim \mathrm{im} \left(H^{0}(B,\mathcal{F}(-p)) \to H^{0}(p,\mathcal{F}(-p)|_{p}) \right) \\
& = \dim M - \dim Y - \rk \FF
\end{align*}
It follows that $\dim M_{p,z}^{\text ram} \leq \dim M - \dim Y - 2$ in a neighborhood of any general point $(s,p)$ in $\mathcal{R}_{0}$.

On the other hand, note that $\mathcal{R}_{0}$ has dimension at most $\sup_{s,p} \{ \dim M_{p,s(p)}^{\text ram} \} + 1 + \dim(Y)$ as we vary $(s,p)$ over general points in $\mathcal{R}_{0}$.  Here the $1$ accounts for the choice of $p \in B$ and the $\dim(Y)$ accounts for the choice of image $s(p) \in Y$.  Combining with the inequality above, we conclude that $\dim(\mathcal{R}_{0}) \leq \dim(M) - 1$.  This shows that $\mathcal{R}_{0}$ cannot map dominantly to $M$.
\end{proof}

We next bound the quantity $q$.

\begin{lemma} \label{lemm:sbound}
Let $Z$ be a smooth projective variety.  Suppose that $M \subset \Mor(B,Z)$ is an irreducible component parametrizing a dominant family of curves on $Z$ and let $W = M_{red}$.  For a general map $s$ parametrized by $M$ consider the composition
\begin{equation*}
V \to H^{0}(B,s^{*}T_{Z})  \to H^{0}(B,(N_{s})_{tf})
\end{equation*}
where $V \subset H^{0}(B,s^{*}T_{Z})$ is the tangent space to $W$ at $s$.  Then the cokernel of the composed map has dimension at most $g(B)\dim(Z) + 5g(B)$.
\end{lemma}

\begin{proof}
Since we have a dominant family the normal sheaf $N_{s}$ is generically globally generated for a general map $s$. 
Since $H^{0}(B,(N_{s})_{tf})$ is a quotient of $H^{0}(B,N_{s})$, the dimension of $\cok(V \to H^{0}(B,(N_{s})_{tf}))$ is bounded above by
\begin{align*}
(h^{0}(B,s^{*}T_{Z}) & - \dim(V))  + \dim(\cok (H^{0}(B,s^{*}T_{Z}) \to H^{0}(B,N_{s}))) 
\end{align*}
and we estimate each piece in turn.  First of all, we have 
\begin{align*}
h^{0}(B,s^{*}T_{Z}) - \dim(V) & \leq h^{0}(B,s^{*}T_{Z}) - \dim(M) \\
& \leq h^{1}(B,s^{*}T_{Z}) \\
& \leq h^{1}(B,T_{B}) + h^{1}(B,N_{s}) \\
& \leq 3g(B)  + g(B) (\dim(Z)-1)
\end{align*}
where the last inequality follows from Lemma \ref{lemm:ggh1bound}.  Second, the cokernel of $H^{0}(B,s^{*}T_{Z}) \to H^{0}(B,N_{s})$ has dimension bounded by $h^{1}(B,T_{B}) \leq 3g(B)$.
\end{proof}

Putting these results together, we obtain the Grauert-Mulich theorem for spaces of morphisms.

\begin{theorem}[Grauert-Mulich] \label{theo:gmformorphisms}
Let $Z$ be a smooth projective variety defined over $\mathbb C$ and let $\mathcal{E}$ be a torsion free sheaf on $Z$ of rank $r$.
Let $M$ be an irreducible component of $\Mor(B,Z)$ and let $ev: \mathcal{U} \to Z$ denote the evaluation map.  

\begin{enumerate}
\item Assume that the composition of $ev$ with the normalization map for $\mathcal{U}$ is dominant with connected fibers and that $ev$ is flat on the preimage of some open subset of $M_{red}$.  
Assume that a general $s: B \to Z$ parametrized by $M$ is birational onto its image.  Then we have
\begin{align*}
\Vert \SP_{Z,[s(B)]}(\mathcal{E}) & - \SP(s^{*}\mathcal{E}) \Vert \leq \\
& \frac{1}{2} \left( 2g(B)^{2}\dim(Z)^{2} + 10g(B)^{2} \dim(Z) + 4g(B) \dim(Z) + 13g(B) + 2 \right)  \rk(\mathcal{E})
\end{align*}
\item Assume that the composition of $ev$ with the normalization map for $\mathcal{U}$ is dominant with connected fibers and that $ev$ is flat on the preimage of some open subset of $M_{red}$.  
Assume that there is some curve $B'$ of genus $\geq 2$ such that the general map $s: B \to Z$ parametrized by $M$ factors through a morphism $s': B' \to Z$ that is birational onto its image.  Then we have 
\begin{align*}
\Vert \SP_{Z,[s(B)]}(\mathcal{E}) & - \SP(s^{*}\mathcal{E}) \Vert \leq \\
& \left( g(B)^{2}\dim(Z)^{2} + 5g(B)^{2} \dim(Z) + 2g(B) \dim(Z) + 8g(B) + 1 \right)  \rk(\mathcal{E})
\end{align*}
\item Let $s: B \to Z$ be a general map parametrized by $M$.  Assume that $\mu^{min}(s^{*}T_{Z}) > 2g(B)$ and that $\dim(Z) \geq 2$. Then
\begin{align*}
\Vert \SP_{Z,[s(B)]}(\mathcal{E}) & - \SP(s^{*}\mathcal{E}) \Vert \leq (3g(B)+1) \rk(\mathcal{E})
\end{align*}
\end{enumerate}
\end{theorem}

\begin{proof}
(1) Let $t$ be the length of the torsion part of $N_{s}$, let $\mathcal{G}$ be the subsheaf of $(N_{s})_{tf}$ generated by global sections, and let $q$ be the dimension of the cokernel of the composition
\begin{equation*}
V \to H^{0}(B,s^{*}T_{Z})  \to H^{0}(B,(N_{s})_{tf})
\end{equation*}
where $V \subset H^{0}(B,s^{*}T_{Z})$ is the tangent space to $M_{red}$ at $s$.

Theorem \ref{theo:butler} shows that $\mu^{max}(M_{\mathcal{G}}^{\vee}) \leq 2g(B) \dim(Z)+2$. Proposition \ref{prop:torsionestimate}.(1) shows that $t \leq 3g(B)$.   By Lemma \ref{lemm:sbound} we have $q \leq g(B)\dim(Z) + 5g(B)$.  We then apply Theorem \ref{theo:generalgm} to obtain the desired statement.

(2) Proposition \ref{prop:torsionestimate}.(2) shows that in this situation $t \leq 6g(B)$.
Then we can obtain the desired bound by repeating the argument for (1) using this new estimate on $t$. 

(3) In this setting the general map $s$ is unramified by Proposition \ref{prop-unramFrank2} so that $N_{s}$ is torsion-free.  By Lemma \ref{lemma:hnfreecurves} $s^{*}T_{Z}$ is globally generated and thus $N_{s}$ is also globally generated.  We have $\mu^{min}(N_{s}) \geq \mu^{min}(s^{*}T_{Z}) > 2g(B)$
so by Theorem \ref{theo:butler} $\mu^{max}(M_{N_{s}}^{\vee}) \leq 2$.  Furthermore $M$ is smooth at a general morphism $s$ and the map $H^{0}(B,s^{*}T_{Z})  \to H^{0}(B,N_{s})$ has cokernel whose dimension is bounded above by $h^{1}(B,T_{B}) \leq 3g(B)$.  

Note that the general fiber of $\mathcal{U} \to M$ is contained in the locus where $ev$ is smooth, and hence also in the flat locus of $ev$.  Thus we have verified the hypotheses of Theorem \ref{theo:generalgm} which gives the desired statement.
\end{proof}

\section{Constructing foliations} 

Suppose that $Z$ is a smooth projective variety and $M$ is an irreducible component of $\Mor(B,Z)$ parametrizing a dominant family of curves.  We would like to understand the slopes in the Harder-Narasimhan filtration of $s^{*}T_{Z}$ for a general map $s: B \to C \subset Z$ parametrized by $M$.  In turn, these slopes control the deformation theory of $s$.  In this section we show that under certain conditions one can identify an algebraic foliation on $Z$ which controls the behavior of these deformations.

We will need the following construction describing the relationship between foliations and relative tangent bundles.

\begin{construction} \label{cons:absolutefoliationrestriction}
Let $Z$ be a smooth projective variety.  Suppose that $\mathcal{F}$ is a foliation on $Z$ that is is induced by a rational map $\rho: Z \dashrightarrow W$ with connected fibers.

Suppose that $s: B \to Z$ is a morphism whose image is contained in the regular locus of $\mathcal{F}$ and goes through a general point of $Z$.  Let $Y$ denote the main component of $B \times_{W} Z$ equipped with the map $g: Y \to Z$.  By \cite[Remark 19]{KSCT07} we can choose a resolution $\mu: \widetilde{Y} \to Y$ and a morphism $\widetilde{s}: B \to \widetilde{Y}$ such that $g \circ \mu \circ \widetilde{s} = s$ and $\widetilde{s}^{*}T_{\widetilde{Y}/B} \cong s^{*}\mathcal{F}$.
\end{construction}

We will also need the following flattening construction.

\begin{construction} \label{cons:flatteningfamilyofcurves}
Let $Z$ be a smooth projective variety and let $W$ be a locally closed subvariety of $\Mor(B,Z)$ parametrizing a dominant family of curves. 
Let $\mathcal{U}_{W}$ denote the universal family over $W$ and let $\mathcal{U}_{W}^{\nu}$ denote the normalization of $\mathcal{U}_{W}$.  Then $\mathcal{U}_{W}^{\nu}$ is equipped with a map $p: \mathcal{U}_{W}^{\nu} \to W$ and an evaluation map $ev_{W}: \mathcal{U}_{W}^{\nu} \to Z$.  We claim there is a birational map $\phi: Z' \to Z$ from a smooth variety $Z'$ and an open subset $W^{\circ} \subset W$ such that the preimage $\mathcal{U}_{W}^{\nu,\circ} := p^{-1}W^{\circ}$ admits a flat morphism $ev': \mathcal{U}_{W}^{\nu,\circ} \to Z'$ satisfying $ev_{W}|_{\mathcal{U}_{W}^{\nu,\circ}} = \phi \circ ev'$.

Indeed, suppose we take a flattening of $ev$, i.e.~a diagram
\begin{equation*}
\xymatrix{
\mathcal{V} \ar[r]^{\widetilde{ev}} \ar[d]_{\psi} & \widetilde{Z} \ar[d]^{\psi_{Z}} \\
\mathcal{U}_{W}^{\nu} \ar[r]_{ev_{W}} & Z
}
\end{equation*}
where $\mathcal{V}$ and $\widetilde{Z}$ are quasi-projective varieties, $\psi$ and $\psi_{Z}$ are projective birational morphisms,
and $\widetilde{ev}$ is flat. (See, e.g., \cite[Theorem 14.143]{GW20} for the version of flattening we use or the original source \cite{RG71}.) Let $\rho: Z' \to \widetilde{Z}$ be a resolution of singularities.  Since $\widetilde{ev}$ is flat, $\mathcal{V}' := \mathcal{V} \times_{\widetilde{Z}} Z'$ is also a quasi-projective variety and the projection map $ev': \mathcal{V}' \to Z'$ is still flat.   
The induced map $\psi': \mathcal{V}' \to \mathcal{U}_{W}^{\nu}$ is still birational.  Since $p$ defines a family of curves, there is an open subset $W^{\circ} \subset W$ such that $p^{-1}W^{\circ}$ is disjoint from every $\psi'$-exceptional center.  Then $W^{\circ}$ has the desired properties.
\end{construction}

We are now prepared to prove the main theorem in this section.

\begin{theorem} \label{theo:absfoliationresult}
Let $Z$ be a smooth projective variety defined over $\mathbb C$.  Fix a positive integer $J \geq 2g(B)+3$.  Suppose $M$ is an irreducible component of $\Mor(B,Z)$ parametrizing a dominant family of morphisms 
and let $ev: \mathcal{U}^{\nu} \to Z$ denote the evaluation map for the normalization of the universal family over $M$.  Assume that either:
\begin{itemize}
\item the general map $s : B \to Z$ parametrized by $M$ is birational onto its image, or
\item there is a smooth projective curve $B'$ of genus $\geq 2$  such that the general map $s: B \to Z$ parametrized by $M$ factors through a morphism $s': B' \to Z$ that is birational onto its image.
\end{itemize}
Suppose furthermore that
\begin{equation*}
\deg(s^{*}T_{Z}) \geq  \dim(Z)(J + 2g(B) + \gamma) + g(B)(\dim(Z)+2) 
\end{equation*}
where we define
\begin{equation*}
\gamma =  \left( g(B)^{2}\dim(Z)^{2} + 5g(B)^{2} \dim(Z) + 2g(B) \dim(Z) + 8g(B) + 1 \right) \dim(Z).
\end{equation*}

Let $\overline{\mathcal U}^\nu$ be a normal projective compactification of $\mathcal{U}^{\nu}$ with a morphism $\overline{ev} : \overline{\mathcal U}^\nu \to X$ extending $ev$.  Let $f: S \to Z$ denote the finite part of the Stein factorization of $\overline{ev}$. Then there is a rational map $\phi: S \dashrightarrow W$ such that the following holds.  Let $s: B \to S$ be a general morphism parametrized by $M$ and define $Y$ to be the main component of $B \times_{W} \widetilde{S}$ in the pullback diagram
\begin{equation*}
\xymatrix{
B \times_{W} \widetilde{S} \ar[r] \ar[d] & \widetilde{S} \ar[d]^{\widetilde{\phi}} \\
B \ar[r]^{\phi \circ s} & W
}
\end{equation*}
where $\widetilde{\phi}: \widetilde{S} \to W$ is a resolution of $\phi$.  We denote by $g$ the morphism $g: Y \to \widetilde{S} \to Z$.  Then there is a resolution $\mu: \widetilde{Y} \to Y$ and a section $\widetilde{s}: B \to \widetilde{Y}$ of the map $\widetilde{Y} \to B$ such that $g \circ \mu \circ \widetilde{s} = s$ and:
\begin{enumerate}
\item The deformations of $\widetilde{s}$ in $\widetilde{Y}$ contain at least $J$ general points of $\widetilde{Y}$.
\item The space of deformations of $\widetilde{s}$ in $\widetilde{Y}$ has codimension in $M$ at most
\begin{equation*}
(\dim(Z)+1)(J + 3g(B) + \gamma+1).
\end{equation*}
\end{enumerate}
\end{theorem}

\begin{proof}
As in Construction~\ref{cons:flatteningfamilyofcurves} we choose a smooth birational model $S' \to S$ that flattens the family $M$.  Since the evaluation map for $M$ factors through $S$, we can take strict transforms of the general maps parametrized by $M$ to obtain a family of maps $s': B \to S'$.  Note that the general map $s'$ is either birational onto its image or it factors through a birational map from a curve of genus $\geq 2$ (possibly different from $B'$).

Recall that the normal sheaf $N_{s'}$ is generically globally generated when we have a dominant family of maps.  In particular, by Lemma \ref{lemm:ggh1bound} we have
\begin{align*}
h^{0}(B,N_{s'}) & \leq h^{0}(B,(N_{s'})_{tf}) + h^{0}(B,(N_{s'})_{tors}) \\
& \leq \deg((N_{s'})_{tf}) + \rk((N_{s'})_{tf})+ h^{0}(B,(N_{s'})_{tors}) = \deg(N_{s'}) + \rk(N_{s'})
\end{align*}
Thus
\begin{align*}
-K_{Z} \cdot s_{*}B + \dim(Z)(1-g(B)) & \leq \dim(M) \leq h^{0}(B,s'^{*}T_{S'}) \\
& \leq h^{0}(B,N_{s'}) + 3 \\
& \leq -K_{S'} \cdot s'_{*}B + (2g(B) - 2) + (\dim(S')-1) + 3
\end{align*}
  Combining with our degree bound, we see
\begin{equation}
\label{equation:degreebounds}
\dim(Z)(J+2g(B) +\gamma) \leq -K_{Z} \cdot s_{*}B - g(B) (\dim(Z)+2) \leq -K_{S'} \cdot s'_{*}B.
\end{equation}

Consider the family of sections of $B \times S' \to B$ corresponding to the morphisms $s': B \to S'$.  Let us assume that this family does not contain $J$ general points of $B \times S'$.  By \cite[Lemma 3.6]{LRT22} we see that
\begin{equation} \label{eq:minslopegbound}
\mu^{min}(s'^{*}T_{S'}) < J + 2g(B) - 1.
\end{equation}
Write the Harder-Narasimhan filtration of $T_{S'}$ with respect to $\alpha := [s'_{*}B]$ as  
\begin{equation*}
0= \FF_0 \subset \FF_1 \subset \dots \subset \FF_k = T_{S'}.
\end{equation*}
Since by (\ref{equation:degreebounds}) we have
\begin{equation*}
\mu_{\alpha}(T_{S'}) = \frac{-K_{S'} \cdot s'_{*}B}{\dim(Z)} \geq J + 2g(B) + \gamma
\end{equation*}
there is some index $i \geq 1$ such that we have $\mu_{\alpha}^{min}(\FF_{i}) \geq J + 2g(B) + \gamma$.  Let $i$ be the maximum index for which this inequality holds.  Applying Theorem \ref{theo:gmformorphisms} to relate the Harder-Narasimhan filtration of $T_{S'}$ to the Harder-Narasimhan filtration of its pullback, we obtain
\begin{align*}
\mu^{min}_{\alpha}(T_{S'}) & \leq \mu^{min}(s'^{*}T_{S'}) + \gamma \\
& < J + 2g(B) + \gamma - 1
\end{align*}
where the second inequality follows from Equation \eqref{eq:minslopegbound}.  This shows that $i < k$.  On the other hand, since $i$ was selected to be as large as possible we must have
\begin{equation*}
\mu^{max}_{\alpha}(T_{S'} / \FF_i) <  J + 2g(B) + \gamma.
\end{equation*}
Since $\FF_{i}$ is a term in the $\alpha$-Harder-Narasimhan filtration of $T_{S'}$ that satisfies $\mu^{min}_{\alpha}(\FF_{i}) > 0$, it is a foliation on $S'$ by Theorem \ref{theo:HNisfoliation}.

By \cite[Theorem 1.1]{CP19} the foliation $\FF_i$ is induced by a rational map $\phi: S' \dashrightarrow W$ that has connected fibers.  Since $i < k$ this rational map is not trivial.  By our flatness construction a general morphism $s'$ parametrized by $M$ will have image contained in the regular locus of $\FF_{i}$.
Let $\widetilde{Y}$ denote a resolution of the main component of $B \times_{W} \widetilde{S}$ as in the statement of the theorem 
and let $\widetilde{s}: B \to \widetilde{Y}$ denote the section chosen as in Construction \ref{cons:absolutefoliationrestriction}.  In particular we have
\begin{equation*}
\widetilde{s}^{*}T_{\widetilde{Y}/B} \cong s^{*}\mathcal{F}_{i}.
\end{equation*}
Theorem \ref{theo:gmformorphisms} implies that
\begin{align*}
\mu^{min}(s^{*}\mathcal{F}_{i}) & \geq \mu_{\alpha}^{min}(\mathcal{F}_{i}) - \gamma \\
& \geq J + 2g(B)
\end{align*}
and so by Proposition \ref{prop:generalimplieshnfree} we see that deformations of $\widetilde{s}$ can go through at least $J$ general points of $\widetilde{\mathcal{Y}}$ verifying (1).  To prove (2), let $N$ denote the space of deformations of $\widetilde{s}$ in $\widetilde{Y}$.  Appealing to Proposition~\ref{prop:domfamilyexpdim}, we see that
\begin{align*}
\dim(M) - \dim(N) & \leq (-K_{S'} \cdot s_{*}B + \dim(Z) + 2g(B)) - (-K_{\widetilde{Y}/B} \cdot \widetilde{s}_{*}B + (\dim(\widetilde{Y}) - 1)(1-g(B))) \\
& = \deg(s^{*}T_{S'}/\FF_{i}) + (\dim(Z) - \dim(\widetilde{Y})+1) + g(B) (\dim(\widetilde{Y}) + 1) \\
& \leq (\dim(Z) - \dim(\widetilde{Y})+1) (J + 2g(B) + \gamma + 1) + g(B) (\dim(\widetilde{Y}) + 1) \\
& \leq (\dim(Z) + 1)(J + 3g(B) + \gamma + 1)
\end{align*}
Since the dimension of the space of sections is birationally invariant, we obtain (2).
\end{proof}

\section{Main results} \label{sect:abscase}

\subsection{Families of curves and Fujita invariants}

We start by relating the Fujita invariant to the existence of families of curves.  The following is an analogue of \cite[Theorem 1.12]{LRT22} (and in fact can be deduced from this result).

\begin{theorem} \label{theo:ainvariantandsectionsabsolutecase}
Let $X$ be a smooth projective Fano variety defined over $\mathbb C$ and let $B$ be a complex smooth projective curve.  Fix a positive rational number $a$.  Fix a positive integer $T$.  There is some constant $\xi = \xi(\dim(X), g(B), a, T)$ with the following property.

Suppose that $Y$ is a smooth projective variety equipped with a morphism $f: Y \to X$ that is generically finite onto its image.   Suppose that $N$ is an irreducible component of $\Mor(B,Y)$ parametrizing a dominant family of curves $C$ on $Y$ which satisfy $-f^{*}K_{X} \cdot C \geq \xi$.  Finally, suppose that
\begin{equation} \label{eq:assumedinequality}
\dim(N) \geq a(-K_{X} \cdot C + \dim(X)(1-g(B))) - T.
\end{equation}
Then
\begin{equation*}
a(Y,-f^{*}K_{X}) \geq a.
\end{equation*}
\end{theorem}

\begin{remark}
Since $\dim(M)$ always has at least the expected dimension, the condition in Equation \eqref{eq:assumedinequality} is implied by the more evocative inequality $\dim(N) \geq a \cdot \dim(M) - T$.
\end{remark}

\begin{proof} 
By \cite[Theorem 1.3]{HL20} there is a rational number $\epsilon > 0$ depending only on $a$ and $\dim(X)$ such that no smooth variety of dimension $\leq \dim(X)$ has Fujita invariant in the range $((1-\epsilon)a,a)$ with respect to any big and nef Cartier divisor.  We define
\begin{equation*}
\xi(\dim(X),g(B),a,T) = 1 + \sup \left\{ 0 , \frac{1}{a\epsilon}\left(g(B) (a\dim(X)+2) + (\dim(Y) - a\dim(X)) + T) \right) \right\}.
\end{equation*}

Since $N$ parametrizes a dominant family of curves on $Y$, Proposition \ref{prop:domfamilyexpdim} yields an inequality
\begin{align*}
\dim(N) \leq -K_{Y} \cdot C + \dim(Y) + 2g(B).
\end{align*}
Combining with Equation \eqref{eq:assumedinequality} and rearranging, we find
\begin{equation*}
(K_{Y} - af^{*}K_{X}) \cdot C \leq g(B) (a\dim(X)+2) + (\dim(Y) -a \dim(X)) + T.
\end{equation*}
Since we are assuming $-f^{*}K_{X} \cdot C \geq \xi$, this inequality implies
\begin{equation*}
(K_{Y} - (1-\epsilon)a f^{*}K_{X}) \cdot C < 0.
\end{equation*}
Since $C$ moves in a dominant family on $Y$, this means that $K_{Y} - (1-\epsilon)a f^{*}K_{X}$ is not pseudo-effective.  In other words, we must have $a(Y,-f^{*}K_{X}) > (1-\epsilon)a$.  But by our choice of $\epsilon$ this implies $a(Y,-f^{*}K_{X}) \geq a$.
\end{proof}

\subsection{Classifying non-free curves}
Our earlier theorems address morphisms $B \to X$ which factor through a curve of genus $\geq 2$, so we will need an additional result to handle the genus $\leq 1$ case.

\begin{theorem} \label{theo:factorthroughrational}
Let $X$ be a smooth projective Fano variety of dimension $\geq 2$ defined over $\mathbb C$ and let $B$ be a complex smooth projective curve. There is a constant $\Theta = \Theta(\dim(X),g(B))$ such that the following results hold.

Let $M$ be an irreducible component of $\Mor(B,Z)$ parametrizing a dominant family of maps of anticanonical degree $\geq \Theta$.
\begin{enumerate}
\item Suppose that a general morphism parametrized by $M$ factors through a morphism $\mathbb{P}^{1} \to X$ that is birational onto its image.  Then the image of the general morphism is an anticanonical conic.
\item There cannot be a curve $B'$ of genus $\geq 1$ such that a general morphism parametrized by $M$ is the composition of a morphism $B \to B'$ of degree $\geq 2$ followed by a morphism $B' \to X$ that is birational onto its image.
\end{enumerate}
\end{theorem}

\begin{proof}
Let $s : B \to X$ be a general map parametrized by $M$.
Set $r = \deg(s^{*}T_{X})$ and define $\Theta = 1 + \sup\{9,12(g(B)\dim(X) + 2g(B)- 2)\}$.

(1) First note that $\dim(M) \geq r + \dim(X) (1-g(B))$.  Suppose that the image of the general morphism $s$ is a rational curve of anticanonical degree $d$.  Since a dominant family of rational curves has the expected dimension, we find: 
\begin{align*}
\dim(M) & \leq \dim \Mor_{r/d}(B,\mathbb{P}^{1}) + (d + \dim(X) - 3) \\
& \leq \left(2\frac{r}{d} + 1 +2g(B)\right) + (d + \dim(X) - 3)
\end{align*}
Comparing, we see that
\begin{equation*}
r \left(1 - \frac{2}{d} \right) \leq d + g(B)\dim(X) + 2g(B) - 2.
\end{equation*}
Note that $2 \leq d$ and $d = r/e$ for some integer $e \geq 2$.  We have the following cases:
\begin{itemize}
\item Suppose $d \geq 5$.  We conclude that $r \leq \frac{5}{3}(d + g(B)\dim(X) + 2g(B) - 2)$.  Since $2d \leq r$, this in turn implies that $r \leq 10(g(B)\dim(X) + 2g(B)- 2)$, contradicting the bound $\Theta \leq r$.
\item Suppose $d = 4$.  Then $r \leq 2(d + g(B)\dim(X) + 2g(B)- 2)$.  If the generic map $B \to \mathbb{P}^{1}$ has degree $2$ then $r = 8$.  Otherwise $r \geq 3d$ and we conclude $r \leq 6(g(B)\dim(X) + 2g(B)- 2)$, contradicting the bound $\Theta \leq r$.
\item Suppose $d = 3$.  Then $r \leq 3(d + g(B)\dim(X) + 2g(B)- 2)$.  If the generic map $B \to \mathbb{P}^{1}$ has degree $2$ or $3$ then $r \leq 9$.  Otherwise $r \geq 4d$ and we conclude $r \leq 12(g(B)\dim(X) + 2g(B)- 2)$, contradicting the bound $\Theta \leq r$.
\end{itemize}
Altogether we see that we must have $d=2$, proving the statement.

(2) As before we have $\dim(M) \geq r + \dim(X) (1-g(B))$.  Suppose that the image of the general morphism $s$ is a curve birational to $B'$ of anticanonical degree $d$.  By Proposition \ref{prop:domfamilyexpdim} a dominant family of maps $B' \to X$ dimension at most $d + \dim(X) + 2g(B')$.  On the other hand, the tangent space to $\Mor_{r/d}(B,B')$ has dimension $1$ (if $B'$ is elliptic) or $0$ (if $g(B') \geq 2$), and this also bounds the dimension of the moduli space.  Altogether we find 
\begin{align*}
\dim(M) & \leq \dim \Mor_{r/d}(B,B') + (d + \dim(X) + 2g(B)) \\
& \leq d + \dim(X) + 2g(B) + 1
\end{align*}
Comparing against the lower bound on $\dim(M)$, we see that
\begin{equation*}
r \leq d + \dim(X)g(B) + 2g(B) + 1
\end{equation*}
Since we must have $r \geq 2d$, we see that $r \leq 2\dim(X)g(B) + 4g(B) + 2$.  But this contradicts the assumption $\Theta \leq r$.  
\end{proof}

We can now prove a classification theorem for non-free curves.

\begin{theorem} \label{theo:maintheoremabscase}
Let $X$ be a smooth projective Fano variety defined over $\mathbb C$ and let $B$ be a complex smooth projective curve.  There are constants $\widetilde{\xi} = \widetilde{\xi}(\dim(X),g(B))$ and $T^{+} = T^{+}(\dim(X),g(B))$ such that the following holds.  Suppose that $M \subset \Mor(B,X)$ is an irreducible component parametrizing non-free maps $s: B \to X$ of anticanonical degree $\geq \widetilde{\xi}$.  Let $\mathcal{U}^{\nu}$ be the normalization of the universal family over $M$ and let $\ev: \mathcal{U}^{\nu} \to X$ be the evaluation map.  Then either:
\begin{enumerate}
\item $ev$ is not dominant.  Then the subvariety $Y$ swept out by the curves parametrized by $M$ satisfies $a(Y,-K_{X}|_{Y}) \geq a(X,-K_{X})$.
\item $ev$ is dominant and the general map parametrized by $M$ is birational onto its image. Let $\overline{\mathcal U}^\nu$ be a normal projective compactification of $\mathcal{U}^{\nu}$ with a morphism $\overline{ev} : \overline{\mathcal U}^\nu \to X$ extending $ev$. Then the finite part $f: Y \to X$ of the Stein factorization of $\overline{ev}$ satisfies
\begin{equation*}
a(Y,-f^{*}K_{X}) = a(X,-K_{X}).
\end{equation*}
Furthermore, there is a rational map $\phi: Y \dashrightarrow Z$ such that the following properties hold.  Let $\widehat{Y}$ denote a smooth projective birational model of $Y$ admitting a morphism $\widehat{\phi}: \widehat{Y} \to Z$ that resolves $\phi$.
\begin{enumerate}
\item Let $F$ denote a general fiber of $\phi$.  Then we have $a(F,-f^{*}K_{X}|_{F}) = a(X,-K_{X})$ and $(F,-f^{*}K_{X}|_{F})$ is adjoint rigid.
\item  Let $\widehat{s}: B \to \widehat{Y}$ denote a map to $\widehat{Y}$ induced by a general point of $M$ and let $W$ denote the main component of $B \times_{Z} \widehat{Y}$.  Then the image in $M$ of the parameter space of deformations of the map $(\mathrm{id},\widehat{s}): B \to W$ has codimension at most $T$ in $M$.
\end{enumerate}

\item $ev$ is dominant and the general map parametrized by $M$ is not birational to its image.  In this case the image of the general map is a rational curve of anticanonical degree $2$.  Thus $ev$ factors rationally through a generically finite map $g: \mathcal{V} \to X$ where $\mathcal{V}$ is a projective model of a universal family $\mathcal U \to N$ of rational curves of anticanonical degree $2$ on $X$ where $N$ is an irreducible open locus of the Hilbert scheme of $X$.  
In particular $a(\mathcal{V},-g^{*}K_{X}) = a(X,-K_{X})$.  
\end{enumerate}
\end{theorem}

\begin{proof}
We first define several constants.  Let $\xi_{1}$ be the constant $\xi(\dim(X),g(B),1,0)$ in Theorem \ref{theo:ainvariantandsectionsabsolutecase}.  By \cite[Theorem 0.2]{KMM92} there is a constant $b$ depending only on $\dim(X)$ such that $-bK_{X}$ is basepoint free.  We next apply \cite[Theorem 7.10]{LRT22} with our choice of $b$, with $a_{rel} = a = 1$, with $E = \tau(\pi,E) = 0$, with $T = 0$, and with $\beta = 0$ to obtain a constant $\Gamma(\dim(X),g(B))$.  Define $T^{+} = (\dim(X)+1)(\Gamma + 3g(B) + \gamma + 2)$ where  
\begin{equation*}
\gamma =  \left( g(B)^{2}\dim(X)^{2} + 5g(B)^{2} \dim(X) + 2g(B) \dim(X) + 8g(B) + 1 \right) \dim(X).
\end{equation*}
Finally set
\begin{equation*}
\widetilde{\xi}(\dim(X),g(B)) = \sup \left\{ \begin{array}{c} \xi_{1}, \xi(\dim(X),g(B),1,(\dim(X)+1)(\Gamma + 3g(B) + \gamma + 2)), \\ \Theta, \dim(X)(\Gamma + 2g(B) + \gamma + 1) + g(B)(\dim(X)+2) \end{array} \right\}
\end{equation*}
where $\xi$ is defined as in Theorem \ref{theo:ainvariantandsectionsabsolutecase} and $\Theta$ is defined as in Theorem \ref{theo:factorthroughrational}. 

(1) Let $\psi: \widetilde{Y} \to Y$ be a resolution and let $N$ parametrize the strict transforms of the curves on $\widetilde{Y}$.  Denote the composition of $\phi$ with the inclusion map by $\widetilde{\psi}: \widetilde{Y} \to X$.  Since $\dim(N) = \dim(M)$, we may apply Theorem \ref{theo:ainvariantandsectionsabsolutecase} to $\widetilde{\psi}$ with the constants $a=1$ and $T=0$.  We conclude that $a(\widetilde{Y}, -\widetilde{\psi}^{*}K_{X}) \geq a(X,-K_{X})$. 
 Since the Fujita invariant is a birational invariant, we conclude (1).

(2)  
Let $\psi: \widetilde{Y} \to Y$ be a resolution and let $N$ parametrize the strict transforms of the curves on $\widetilde{Y}$.  Denote the composition of $\psi$ with $f$ by $\widetilde{f}: \widetilde{Y} \to X$.  Since $\dim(N) = \dim(M)$, we may apply Theorem \ref{theo:ainvariantandsectionsabsolutecase} to $\widetilde{f}$ with the constants $a=1$ and $T=0$.  We conclude that $a(\widetilde{Y}, -\widetilde{f}^{*}K_{X}) \geq a(X,-K_{X})$; since $\widetilde{f}$ is dominant the equality must be achieved.  Since the Fujita invariant is a birational invariant, we also have $a(Y, -f^{*}K_{X}) = a(X,-K_{X})$.

Since our curves have anticanonical degree $\geq \widetilde{\xi}$ and the general morphism $s$ is birational onto its image, we can apply Theorem \ref{theo:absfoliationresult} to $X$ and $M$ with $J = \sup\{2g(B) + 3,\Gamma + 1\}$.  Since by assumption the curves parametrized by $M$ are non-free the corresponding sections cannot go through $J$ general points of $Y \times B$ (see \cite[Proposition 3.7]{LRT22}).   Thus Theorem \ref{theo:absfoliationresult} yields a non-trivial rational map $\phi: Y \dashrightarrow Z$ with the following property.  Suppose $\widehat{\phi}: \widehat{Y} \to Z$ is a resolution of $\phi$.  For a general map $\widehat{s}: B \to \widehat{Y}$ parametrized by $M$, let $W$ denote the main component of $B \times_{W} \widehat{Y}$.  Then there is a resolution $\widetilde{W} \to W$ and a section $\widetilde{s}: B \to \widetilde{W}$ such that the deformations of $\widetilde{s}$ contain at least $\Gamma+1$ general points of $\widetilde{W}$.  Furthermore the space of deformations of $\widetilde{s}$ in $\widetilde{W}$ has codimension at most $T^{+}$ in $M$.  Letting $h: \widetilde{W} \to X$ denote the induced map and $\widetilde{W}_{\eta}$ denote the generic fiber of $h$, our choice of $\xi$ shows that $a(\widetilde{W}_{\eta},-h^{*}K_{X}|_{\widetilde{W}_{\eta}}) = 1$ and the conclusion of \cite[Theorem 7.10]{LRT22} shows that $(\widetilde{W}_{\eta},-h^{*}K_{X}|_{\widetilde{W}_{\eta}})$ is adjoint rigid.  Since the Fujita invariant and the Iitaka dimension are constant for general fibers of the map $W \to B$ by invariance of plurigenera (see \cite[Theorem 4.3]{LT17}) we conclude that a general fiber $F$ of $W \to B$ has Fujita invariant $1$ and is adjoint rigid with respect to the pullback of $-K_{X}$.  But since $s$ is general we see that $F$ is also a general fiber of $\widehat{\phi}$, finishing the proof.

(3) It only remains to consider the case when a general map $s$ is the composition of a morphism $B \to B'$ of degree $\geq 2$ and a morphism $s'': B' \to X$ that is birational to the image.
Theorem \ref{theo:factorthroughrational} combined with our definition of $\widetilde{\xi}$ shows the case $g(B') \geq 1$ is not possible.  If $g(B') = 0$ then the rational factoring through $\mathcal{V}$ is an immediate consequence of Theorem \ref{theo:factorthroughrational}.  Note that a rational curve of anticanonical degree $2$ on $X$ satisfies $a(C,-K_{X}) = 1 = a(X,-K_{X})$.  Since $\mathcal{V}$ is covered by such conics, we have $a(\mathcal{V},-g^{*}K_{X}) \geq a(X,-K_{X})$ by \cite[Lemma 4.8]{LST18}.  The reverse inequality $a(\mathcal{V},-g^{*}K_{X}) \leq a(X,-K_{X})$ follows from the Riemann-Hurwitz formula as in \cite[Lemma 4.7]{LST18}.  
\end{proof}

\section{Boundedness statements in the absolute case}
\label{sec:boundedness}

The goal of this section is to prove the following boundedness theorem.

\begin{theorem}
\label{theorem:non-dominant}
Let $X$ be a smooth projective Fano variety and let $B$ be a smooth projective curve, both are defined over $\mathbb C$.  There is a proper closed subset $V \subsetneq X$ such that if $M \subset \Mor(B,X)$ is an irreducible component parametrizing a non-dominant family of curves then every curve parametrized by $M$ is contained in $V$.
\end{theorem}

This theorem is almost an immediate consequence of the analogous boundedness result in the relative setting, \cite[Theorem 8.10]{LRT22}.  There is one important issue: there is a particular closed set of $X \times B$ used in \cite{LRT22} and we must verify that this closed set does not dominate $X$ under the projection map.  This claim is true, but unfortunately a careful verification takes some work.  We make this verification in the rest of the section, giving the precise statement in Theorem~\ref{theo:positivebounded}.

For the rest of this section $k$ denotes an algebraic closed field of characteristic $0$ and $B$ denotes a smooth projective curve defined over $k$. We denote the function field of $B$ by $F$.  We will freely use the language and constructions of \cite{LST18} and \cite{LRT22} when we give the appropriate reference.

\subsection{Lemma 8.3 of \cite{LST18}}

First we recall some definitions from \cite{LST18}:
\begin{definition}
Let $F$ be a field of characteristic $0$. Let $X$ be a smooth geometrically uniruled projective variety defined over $F$ and $L$ be a big and nef $\mathbb Q$-divisor on $X$. We define the $b$-invariant for $(X, L)$ as
\[
b(F, X, L) := \text{codimension of the supported face of $\overline{\mathrm{Eff}}^1(X)$ containing $a(X, L)L + K_X$.}
\]
When $X$ is singular, we take a resolution $\beta : \widetilde{X} \to X$ and define the $b$-invariant as
\[
b(F, X, L):= b(F, \widetilde{X}, \beta^*L).
\]
This is well-defined due to \cite[Proposition 2.10]{HTT15}.
When $F$ is algebraically closed, we drop $F$ from $b(F, X, L)$ and simply denote it by $b(X, L)$.
\end{definition}

\begin{definition}[{\cite[Definition 7.1 and Definition 8.2]{LST18}}]
First let us assume that our ground field $F$ is an algebraically closed field of characteristic $0$.
A good family of adjoint rigid varieties is a morphism $p : \mathcal U \to W$ of quasi-projective smooth varieties and $p$-relatively big and nef $\mathbb Q$-divisor $L$ on $\mathcal U$ satisfying the following properties:
\begin{enumerate}
\item The morphism $p$ is projective, smooth, and surjective with irreducible fibers;
\item for any closed point $w \in W$ and the corresponding fiber $\mathcal U_w$ above $w$, $a(\mathcal U_w, L|_{\mathcal U_w})$ is constant and positive, and $(\mathcal U_w, L|_{\mathcal U_w})$ is adjoint rigid;
\item $b(\mathcal U_w, L|_{\mathcal U_w})$ is also constant for any closed point $w \in W$, and;
\item Let $Q$ denote the union of all divisors $D$ in fibers $\mathcal U_w$ such that $a(D, L|_D) > a(\mathcal U_w, L|_{\mathcal U_w})$. Then $Q$ is closed in $\mathcal U$ and flat over $W$. Moreover if we set $\mathcal V = \mathcal U \setminus Q$, there is a projective birational morphism $\phi : \mathcal U' \to \mathcal U$ that is an isomorphism over $\mathcal V$ such that $\mathcal U'$ is smooth over $W$ and $\mathcal U' \setminus \mathcal V$ is a strict normal crossings divisor relative to $W$.
\end{enumerate}
Now let us assume that $F$ is an arbitrary field of characteristic $0$. A projective morphism $p : \mathcal U \to W$ is a good family of adjoint rigid varieties if the base change of $p$ to $\overline{F}$ is a good family of adjoint rigid varieties.
\end{definition}

\begin{definition}[{\cite[p.1405]{LST18}}]
Again let us assume that our ground field $F$ is an algebraically closed field of characteristic $0$.
A good morphism of good families of adjoint rigid varieties is a diagram
\begin{equation*}
\xymatrix{ \mathcal Y \ar[r]^{f} \ar[d]_{q} &  \mathcal {U} \ar[d]_{p} \\
T \ar[r] & W}
\end{equation*}
and $p$-relatively big and nef $\mathbb Q$-divisor $L$ on $\mathcal U$ such that $p, q$ are good families of adjoint rigid varieties with respect to $L$ and $f^*L$ respectively, the relative dimensions of $p, q$ are equal, and for any closed point $t \in T$, we have $a(\mathcal Y_t, f^*L|_{\mathcal Y_t}) = a(\mathcal U_w, L|_{\mathcal U_w})$.

When $F$ is an arbitrary field of characteristic $0$, a diagram 
\begin{equation*}
\xymatrix{ \mathcal Y \ar[r]^{f} \ar[d]_{q} &  \mathcal {U} \ar[d]_{p} \\
T \ar[r] & W}
\end{equation*}
is a good morphism of good families if its base change to $\overline{F}$ is a good morphism.
\end{definition}

The following lemma is essentially \cite[Lemma 8.3]{LST18}. The only difference is that we perform every construction in the absolute setting:

\begin{lemma} \label{lemm: finitelymanycoversovernf}
Let $X$ be a uniruled smooth projective variety defined over $k$ 
and $L$ be a big and nef $\mathbb Q$-divisor on $X$.
Let $\mathcal X$ be the base change of $X$ to $F$ and $L_F$ be the base change of $L$ to $F$. 
Let $p : U \rightarrow W$ be a surjective morphism between projective $k$-varieties where $U$ is equipped with a morphism $s : U \to X$.  Let $p_F : \mathcal U \rightarrow \mathcal W$ be the base change of $p$ to $F$ with the morphism $s_F : \mathcal U \to \mathcal X$.

Suppose that there exists a Zariski open subset $W^\circ \subset W$ such that $p: U^\circ \to W^\circ$ is a good family of adjoint rigid varieties over $k$ (here $U^\circ$ denotes the preimage of $W^\circ$) and that any fiber over $W^\circ$ has the same $a$-invariant with respect to $s^{*}L$ as $X$ has with respect to $L$. Then there exist a proper closed subset $R \subsetneq X$ and a finite set of dominant generically finite morphisms  $\{ f_{j}: Y_{j} \to {U} \}$ defined over $k$ that can be fit into commutative diagrams
\begin{equation*}
\xymatrix{ {Y}_{j} \ar[r]^{f_{j}} \ar[d]_{q_{j}} &  {U} \ar[d]_{p} \\
T_{j} \ar[r] & W}
\end{equation*}
such that the following holds.
Let $\mathcal Y_{j}, \mathcal T_{j}, f_{j, F}, q_{j, F}$ be the base changes of the corresponding objects to $F$. Then:
\begin{enumerate}
\item both $Y_j$ and $T_j$ are projective varieties, $Y_j$ is smooth, $T_j$ is normal, and $q_{j}: {Y}_{j} \to T_{j}$ is generically a good family of adjoint rigid varieties;
\item the canonical model for $a(X, L)f_j^*s^*L + K_{Y_j}$ is a morphism and this morphism agrees with $q_j$ over some open set of $T_j$;  
\item the morphism $T_j \rightarrow W$ is dominant, finite, and Galois;
\item we have $\mathrm{Bir}(Y_{j}/X) = \mathrm{Aut}(Y_{j}/X)$;
\item
Assume that $q: \mathcal{Y} \to \mathcal T$ is a projective surjective morphism of varieties over $F$ where $\mathcal{Y}$ is smooth and geometrically integral and that we have a diagram
\begin{equation*}
\xymatrix{ \mathcal{Y} \ar[r]^{f} \ar[d]_{q} &  \mathcal{U} \ar[d]_{p_F} \\
\mathcal T \ar[r]^{g} & \mathcal W}
\end{equation*}
 satisfying the following properties:
\begin{enumerate}
\item There is a Zariski open subset $\mathcal T' \subset \mathcal T$ such that $\mathcal{Y}$ is a good family of adjoint rigid varieties over $\mathcal T'$ and the map $f: q^{-1}(\mathcal T') \to \mathcal{U}$ has image in $\mathcal{U}^{\circ}$ and is a good morphism of good families.
\item There exists a rational point $y \in \mathcal{Y}(F)$ such that $s_F \circ f(y) \not \in \mathcal R$ where $\mathcal R$ is the base change of $R$ to $F$.
\end{enumerate}
Then for some index $j$ there will be a twist $f_{j}^\sigma : \mathcal Y_j^\sigma \rightarrow \mathcal U$ over $F$ such that $f(y) \in f_{j}^{\sigma}(\mathcal{Y}_{j}^{\sigma}(F))$.  Furthermore, there exists a dominant generically finite map $\widetilde{\mathcal T} \to \mathcal T$ such that the main component $\widetilde{q}: \widetilde{\mathcal{Y}} \to \widetilde{\mathcal T}$ of the base change of $q$ by $\widetilde{\mathcal T} \to \mathcal T$
satisfies that the induced map $\widetilde{f}: \widetilde{\mathcal{Y}} \to \mathcal{U}$ factors rationally through $f_{j}^{\sigma}$ in a way that a general geometric fiber of $\widetilde{q}$ maps birationally to a geometric fiber of the map $q_{j}^{\sigma}: \mathcal{Y}_{j}^{\sigma} \to \mathcal T_{j}^{\sigma}$.
\end{enumerate}
\end{lemma}

\begin{proof}
This lemma follows from the proof of \cite[Lemma 8.3]{LST18}. Indeed, one can perform every construction in Steps 1 and 2 of the proof of \cite[Lemma 8.3]{LST18} over $k$. Their base changes to $F$ will satisfy the universal property (5) which can be justified by Step 3 of the proof of \cite[Lemma 8.3]{LST18}.
\end{proof}

\subsection{Section 2.5 of \cite{LRT22}}

We modify \cite[Construction 2.17]{LRT22}.

\begin{construction} \label{cons:rigidsubvarieties}
Let $X$ be a uniruled smooth projective $k$-variety and let $L$ be a big and nef $\mathbb{Q}$-Cartier divisor on $X$.  It follows from \cite[Theorem 4.19]{LST18} that there are a proper closed subset $V$, finitely many projective varieties $W_i \subset \mathrm{Hilb}(X)$, proper families $p_{i}: U_{i} \to W_{i}$ where $U_i$ is a smooth birational model of the universal family $U_i' \to W_i$, and dominant generically finite morphisms $s_i :  U_i \to X$
such that
\begin{itemize}
\item a general fiber $Z$ of $p_{i} : U_{i} \rightarrow W_{i}$ is a smooth uniruled projective variety which is mapped birationally by $s_{i}$ onto the subvariety of $X$ parametrized by the corresponding point of $\Hilb(X)$ and it also satisfies $a(Z,s_{i}^{*}L|_{Z}) = a(X,L)$ and is adjoint rigid with respect to $s_{i}^{*}L|_{Z}$;
and
\item for every subvariety $Y \subset X$ not contained in $\mathbf{B}_{+}(L)$ which satisfies $a(Y,L|_{Y}) \geq a(X,L)$ and which is adjoint rigid with respect to $L$, either $Y$ is contained in $V$ or there is some index $i$ and a smooth fiber of $p_{i}$ that is mapped birationally to $Y$ under the map $s_{i}$.
\end{itemize}
\end{construction}

The following theorem is essentially \cite[Theorem 2.18]{LRT22} but stated in the absolute setting.

\begin{theorem} \label{theo:ainvboundedandtwists}
Let $X$ be a uniruled smooth projective $k$-variety and let $L$ be a big and nef $\mathbb{Q}$-Cartier divisor on $X$.  Denote by $\{ p_{i}: {U}_{i} \to W_{i}\}$ the finite set of families equipped with maps $s_{i}: U_{i} \to X$ and by $V$ the closed subset of Construction \ref{cons:rigidsubvarieties}.  There are a closed set $R \subset X$ and finitely many smooth projective varieties $Y_{i, j}$ equipped with dominant morphisms $r_{i, j}: Y_{i, j} \to T_{i, j}$ with connected fibers and dominant morphisms $h_{i, j}: Y_{i, j} \to U_i$ forming commuting diagrams
\begin{equation*}
\xymatrix{ {Y}_{i, j} \ar[r]^{h_{i, j}} \ar[d]_{r_{i, j}} &  {U}_i \ar[d]_{p_{i}} \\
T_{i, j} \ar[r]_{t_{i, j}} & W_i}
\end{equation*}
that satisfy the following properties.
Let us denote the base changes of $$X, L, U_i, s_i, Y_{i, j}, T_{i, j}, r_{i, j}, h_{i, j}, R$$ to $F$ by $$\mathcal X, L_F, \mathcal U_i, s_{i, F}, \mathcal Y_{i, j}, \mathcal T_{i, j}, r_{i, j, F}, h_{i, j, F}, \mathcal R$$ respectively. Then we have
\begin{enumerate}
\item each map $h_{i, j}$ is generically finite and $f_{i, j} = s_i \circ h_{i, j}$ is not birational; 
\item $t_{i, j}$ is a finite Galois cover and $T_{i, j}$ is normal;
\item we have $\mathrm{Bir}(Y_{i, j}/U_i) = \mathrm{Aut}(Y_{i, j}/U_i)$;
\item every twist $\mathcal Y_{i, j}^{\sigma}$ of $\mathcal Y_{i, j}$ over $\mathcal U_i$ admits a morphism $r_{i, j, F}^{\sigma}: \mathcal Y_{i, j}^{\sigma} \to \mathcal T_{i,j}^{\sigma}$ which is a twist of $r_{i, j, F}$;
\item we have $a(Y_{i, j},f_{i, j}^{*}L) = a(X,L)$;
\item suppose that $\mathcal Y$ is a geometrically integral smooth projective variety and that $f: \mathcal Y \to \mathcal X$ is a morphism that is generically finite onto its image but not birational such that $a(\mathcal Y,f^{*}L_F) \geq a(\mathcal X,L_F)$.  Suppose furthermore that $y \in \mathcal Y(F)$ satisfies $f(y) \not \in \mathcal R$.  Then:
\begin{enumerate}
\item there are indices $i,j$ and a twist $h_{i, j, F}^{\sigma}: \mathcal Y_{i, j}^{\sigma} \to \mathcal U_i$ of $h_{i, j, F}$ such that $f(y) \in s_{i, F}(h_{i, j, F}^{\sigma}(\mathcal Y_{i, j, F}^{\sigma}(F)))$, and
\item if $(\mathcal Y,f^{*}L_F)$ is adjoint rigid then furthermore $f$ factors rationally through $h_{i, j, F}^{\sigma}$ and $f$ maps $\mathcal Y$ birationally to a fiber of $r_{i, j, F}^{\sigma}$.
\end{enumerate}
\end{enumerate}
\end{theorem}

\begin{proof}
There are two inputs into \cite[Theorem 2.18]{LRT22}: \cite[Lemma 8.3]{LST18} and \cite[Construction 2.17]{LRT22}.  Lemma \ref{lemm: finitelymanycoversovernf} is a version of \cite[Lemma 8.3]{LST18} in the absolute setting.  It is clear that the closed set $\mathcal{V}$ of $X \times B$ constructed in \cite[Construction 2.17]{LRT22} is obtained by base change from a closed subvariety $V \subset X$.  With these changes the proof of \cite[Theorem 2.18]{LRT22} works with no issues.
\end{proof}

\subsection{Section 8.1 of \cite{LRT22}}

Here we improve results from \cite[Section 8]{LRT22} in the absolute setting.  Given a smooth projective curve $B$, we denote by $\eta$ its generic point; given a morphism $\mathcal{X} \to B$, we denote by $\mathcal{X}_{\eta}$ its generic fiber. We recall the following definition from \cite{LRT22}:

\begin{definition}[{\cite[Definition 3.1]{LRT22}}]
We say that a morphism $\pi: \mathcal{Z} \to B$ is a good fibration if:
\begin{enumerate}
\item $\mathcal{Z}$ is a smooth projective variety,
\item $B$ is a smooth projective curve, and
\item $\pi$ is flat and has connected fibers.
\end{enumerate}
\end{definition}

We then perform constructions from \cite[Section 8]{LRT22} in the absolute setting:

\begin{construction} \label{cons:allsubvarieties}
Let $X$ be a uniruled smooth projective variety defined over $k$ and $L$ be a big and semiample Cartier divisor on $X$. Set $a = a(X, L)$ and denote $X \times B$ by $\mathcal X$ and the pullback of $L$ to $\mathcal X$ by $L_B$.

Applying Construction \ref{cons:rigidsubvarieties} to $X$ we obtain a proper closed subset $V \subset X$ and a finite collection of families $p_{i} : U_{i} \to {W}_{i}$ whose smooth fibers are birational to closed subvarieties of $X$. Let $\mathcal U_i$ and $\mathcal W_i$ be $U_i \times B$ and $W_i\times B$ with the natural map $p_{i, B}: \mathcal{U}_{i} \to \mathcal{W}_{i}$ and let $\mathcal V$ denote $V\times B$. Let $\mathfrak{W}_i$ be $\Sec(\mathcal{W}_i/B)$. We define $\mathfrak W = \sqcup_i \mathfrak W_i$.
Let $W_i^\circ \subset W_i$ be a Zariski open subset such that over $W_i^\circ$, $p_i|_{p_i^{-1}(W_i)}$ is smooth.
We first shrink $\mathfrak{W}$ so that the generic point of every section parametrized by $\mathfrak{W}$ is contained in $W_{i, \eta}^\circ$. We enlarge $V$ by adding the images in $X$ of the fibers of $p_i$ over $W_i \setminus W_i^\circ$. While doing so, we continue to let $\mathcal V$ denote $V \times B$.
\end{construction}

\begin{construction} \label{cons:allfamilies}

Let $\sqcup_{G} \mathcal H(G, B)$ be the Hurwitz stack parametrizing pairs $(C \to B, \psi)$ where $C \to B$ is a Galois cover from a smooth projective curve $C$ and $\psi : \mathrm{Gal}(C/B) \cong G$ is an isomorphism of groups. (See \cite{Wewers} for the construction of such a stack as a Deligne-Mumford stack.) Fix an \'etale covering $\sqcup_G \mathcal H_G \to \sqcup_G \mathcal H(G, B)$ from a scheme.

Let $X$ be a uniruled smooth projective variety defined over $k$ and $L$ be a big and semiample Cartier divisor on $X$. Set $a = a(X, L)$.  We denote $X \times B$ by $\mathcal X$ and the pullback of $L$ to $\mathcal X$ by $L_B$.
Let $\mathfrak{Z} \to \mathfrak{W} \times B$ be the morphism constructed in Construction \ref{cons:allsubvarieties}.

By Theorem \ref{theo:ainvboundedandtwists} we obtain a closed set $R \subset {X}$ and a finite set of smooth projective $k$-varieties ${Y}_{i,j}$ equipped with dominant generically finite morphisms $h_{i, j}: {Y}_{i, j} \to U_{i}$ and dominant morphisms $r_{i,j}: Y_{i,j} \to T_{i,j}$.  Let $V$ be the union of $R$ with the closed set from Construction~\ref{cons:allsubvarieties}.
Then we enlarge $V$ by adding $s_i(B_{i, j})$ where $B_{i, j}$ is the union of the irreducible components of the branch locus of $h_{i, j}$.  We further enlarge ${V}$ by adding the Zariski closure of the union of the images of the fibers of $r_{i, j}$ which fail to be smooth, fail to have the same $a$-invariant as $Y_{i, j}$, or fail to be adjoint rigid.  We denote $V \times B$ by $\mathcal V$.

We then exactly repeat the remaining steps in \cite[Construction 8.4]{LRT22}.  The result is a family $\mathfrak{F} \to \mathfrak{S} \times B$ whose base is a countable union of finite type schemes and a morphism $g: \mathfrak{F} \to \mathfrak S \times \mathcal{X}$ such that
\begin{enumerate}
\item the fiber $\mathfrak{F}_{s}$ is a normal projective $B$-variety such that $\mathfrak F_s \to B$ has connected fibers for every closed point $s \in \mathfrak{S}$ ;
\item the map $g_{s}: \mathfrak{F}_{s} \to \mathcal{X}$ is a $B$-morphism that is generically finite onto its image and the corresponding morphism $\mathfrak F_s \to \mathfrak Z_w$ is a dominant finite morphism for every closed point $s \in \mathfrak{S}$;
\item we have $a(\mathfrak{F}_{s,\eta},g_{s}^{*}L|_{\mathfrak{F}_{s,\eta}}) = a$ and $(\mathfrak{F}_{s,\eta},g_{s}^{*}L|_{\mathfrak{F}_{s,\eta}})$ is adjoint rigid for every closed point $s \in \mathfrak{S}$, 
\item moreover if $\mathcal{Y}$ is a good fibration over $B$ and $f: \mathcal{Y} \to \mathcal{X}$ is a generically finite $B$-morphism such that  $a(\mathcal{Y}_{\eta},f^{*}L|_{\mathcal{Y}_{\eta}}) = a$ and $(\mathcal{Y}_{\eta},f^{*}L|_{\mathcal{Y}_{\eta}})$ is adjoint rigid, either the map $f$ is birationally equivalent to $g_{s}$ for some closed point $s$ in our family or $f(\mathcal{Y}_{\eta}) \subset \mathcal V_{\eta}$.
\end{enumerate}
We also have a family $\mathfrak Y \to \mathfrak D \times \mathcal X$ parametrizing integral models $h^\sigma_{i, j}:\mathcal Y^\sigma_{i, j} \to \mathcal U_i$ of twists $h^\sigma_{i, j, \eta}: \mathcal Y^\sigma_{i, j, \eta} \to \mathcal U_{i, \eta}$.
\end{construction}

The following theorem is essentially \cite[Theorem 8.7]{LRT22}. The only difference is that we can take our proper closed subset to be the base change of a proper closed subset on $X$.

\begin{theorem} \label{theo:positivebounded}
Let $X$ be a uniruled smooth projective variety defined over $\mathbb C$ and $L$ be a big and semiample Cartier divisor on $X$. Set $a = a(X, L)$.  Denote $X \times B$ by $\mathcal X$ and the pullback of $L$ to $\mathcal X$ by $L_B$.

Fix a rational number $\beta$.  Fix a positive integer $T$.  Fix a positive integer $b > a$ such that $bL'$ defines a basepoint free linear series.  There is:
\begin{itemize}
\item a constant $\xi^{\dagger} = \xi^{\dagger}(\dim(\mathcal{X}), g(B), a, T, \beta,b)$,
\item a proper closed subset ${V} \subset {X}$, and
\item a bounded family of smooth projective varieties $q: \widehat{\mathfrak{F}} \to \widehat{\mathfrak{S}}$ equipped with $\widehat{\mathfrak{S}}$-morphisms $p: \widehat{\mathfrak{F}} \to \widehat{\mathfrak{S}} \times B$ and $g: \widehat{\mathfrak{F}} \to \widehat{\mathfrak{S}} \times \mathcal{X}$
\end{itemize}
which have the following properties:
\begin{enumerate}
\item the morphism $\widehat{\mathfrak F}_s \to B$ is a good fibration for every closed point $s \in \widehat{\mathfrak{S}}$;
\item the morphism $g_{s}: \widehat{\mathfrak{F}}_{s} \to \mathcal{X}$ is a $B$-morphism that is  generically finite onto its image for every closed point $s \in \widehat{\mathfrak{S}}$ ;
\item the composition of $g|_{\widehat{\mathfrak{F}}_{i}}: \widehat{\mathfrak{F}}_{i} \to \widehat{\mathfrak{S}} \times \mathcal{X}$ with the projection $\widehat{\mathfrak{S}} \times \mathcal X \to \mathcal X \to X$ is dominant for every irreducible component $\widehat{\mathfrak{F}}_{i}$ of $\widehat{\mathfrak{F}}$;
\item we have $a(\widehat{\mathfrak{F}}_{s,\eta},g_{s}^{*}L_B|_{\widehat{\mathfrak{F}}_{s,\eta}}) = a({X},L)$ and $(\widehat{\mathfrak{F}}_{s,\eta},g_{s}^{*}L_B|_{\widehat{\mathfrak{F}}_{s,\eta}})$ is adjoint rigid for every closed point $s \in \widehat{\mathfrak{S}}$.
\item Assume that $\psi: \mathcal{Y} \to B$ is a good fibration equipped with a $B$-morphism $f: \mathcal{Y} \to \mathcal{X}$ that is generically finite onto its image and satisfies $a(\mathcal{Y}_{\eta},f^{*}L_B|_{\mathcal{Y}_{\eta}}) \geq a$.  Suppose that $N$ is an irreducible component of $\Sec(\mathcal{Y}/B)$ parametrizing a dominant family of sections $C$ on $\mathcal{Y}$ which satisfy $f^{*}L_B \cdot C \geq \xi$ and $f^{*}(K_{\mathcal{X}/B} + aL_B) \cdot C \leq \beta$.  Let $M \subset \Sec(\mathcal{X}/B)$ be the irreducible component containing the pushforward of the sections parametrized by $N$.  Finally, suppose that
\begin{equation*}
\dim(N) \geq \dim(M) - T.
\end{equation*}

For a general section $C$ parametrized by $N$, either:
\begin{itemize}
\item $f(C)$ is contained in $\mathcal{V} = V \times B$, or
\item there exist an irreducible component $\widehat{\mathfrak{F}}_{i}$ of $\widehat{\mathfrak{F}}$ and an irreducible component $N'$ of $\Sec(\widehat{\mathfrak{F}}_{i}/B)$ parametrizing a dominant family of sections on $\widehat{\mathfrak{F}}_{i}$ such that $f(C)$ is the image of a section $C'$ parametrized by $N'$ and if $\widehat{\mathfrak{F}}_{i,s}$ denotes the fiber containing $C'$ then the strict transform of $C'$ in a resolution of $\widehat{\mathfrak{F}}_{i,s}$ is HN-free.
\end{itemize}
\end{enumerate}
\end{theorem}

\begin{proof}
 We let $L'$ denote the sum of $L_B$ and the pullback of an ample line bundle of degree $1$ from $B$. Note that $L'$ is big and semiample on $\mathcal X$.
\textbf{Step 1:} Let $d$ be as in Step 1 of the proof of \cite[Theorem 8.7]{LRT22}.  Let $\sqcup_{G} \mathcal H(G, B)$ be the Hurwitz stack and fix an \'etale covering $\sqcup_{G, |G| \leq d} \mathcal H_G \to \sqcup_{G, |G| \leq d} \mathcal H(G, B)$ by a scheme. We will work over this base for the entire proof.

Note that the divisor $E = 0$ satisfies the condition of \cite[Proposition 7.1]{LRT22}.  Define $\xi = \xi(\dim(\mathcal{X}), g(B), 0, 1, a, T, \beta+a,b)$ as in \cite[Theorem 7.6]{LRT22}.  We then choose 
$$
\xi^{+} = \xi^{+}(\dim(\mathcal{X}), g(B), 0, 1, a, T, \beta + a, b) \quad \text{ and } \quad T^{+} = T^{+}(\dim(\mathcal{X}), g(B), 0, 1, a, T, \beta+a,b)
$$ 
as in \cite[Corollary 7.11]{LRT22}.  Define $\daleth = \daleth(\dim(\mathcal{X}), g(B), 1, a, T, \beta + a,b)$ as in \cite[Theorem 8.1]{LRT22}.  Finally we define $\xi^{\dagger} = \sup \left\{ \xi, \xi^{+} \right\}$.

Since $L$ is big and semiample, there is a closed subvariety ${V}_{1} \subset X$ such that the family of subvarieties of ${X}$ that are not contained in ${V}_{1}$ and have $L$-degree $\leq \daleth$ is bounded. By \cite[Theorem 4.18.(2)]{LST18} there is a closed sublocus ${V}_{2} \subset {X}$ that contains all subvarieties with larger generic Fujita invariant. Let $V_3$ be the exceptional closed set from Construction~\ref{cons:allfamilies}. We start by setting ${V}$ to be the union of ${V}_{1}$, ${V}_{2}$, and $V_3$; we will later enlarge it.

Let $\mathfrak{Z}_i \to \mathfrak{W}_i\times B$ be the families in Construction \ref{cons:allsubvarieties}.  Then there is a finite-type subscheme $\mathfrak{R}_i \subset \mathfrak{W}_i$ parametrizing those varieties whose images in $\mathcal{X}$ have $L'$-degree $\leq \daleth$ and are not contained in $\mathcal V$.  We define $\mathfrak{Z}_{i, \mathfrak R_i} \to \mathfrak{R}_i$ as the universal family over $\mathfrak{R}_i$. Set $\mathfrak R = \sqcup_i \mathfrak R_i$.

Let $\mathfrak{Y} \to \mathfrak{T} \to \mathfrak{D}$, $\mathfrak T' \to \mathfrak M$, $\mathfrak{F} \to \mathfrak{S}$ and $g: \mathfrak{F} \to \mathcal{X} \times B$ be defined as in Construction \ref{cons:allfamilies}. 
 We let $\mathfrak{S}'$ denote the sublocus of $\mathfrak{S}$ consisting of maps $g_{s}$ whose image is a member of our fixed bounded family $\mathfrak{Z}_{\mathfrak R} \to \mathfrak{R}$ and denote by $\mathfrak{F}' \to \mathfrak{S}'$ the corresponding family.

\textbf{Step 2:}  We next claim that there is a morphism $\mathfrak{Q}\to \mathfrak S'\subset \mathfrak S$ such that $\mathfrak Q$ is of finite type over $\mathbb C$ and for every map $g_{s}$ parametrized by $\mathfrak{S}'$ the map $g_{s,\eta}$ is a twist of the generic fiber of a map parametrized by $\mathfrak{Q}$.  Indeed, this follows from the discussion of Step 2 of \cite[Theorem 8.7]{LRT22} without any modification.

\textbf{Step 3:} Next we define an integer $t$ as in the discussion of Step 3 of \cite[Theorem 8.7]{LRT22}.

\textbf{Step 4:} Lemma 6.3 and Corollary 6.13 of \cite{LRT22}  show that as we vary the closed point $q \in \mathfrak{Q}$ the set of twists of $h_{q}: \mathfrak{P}_{q} \to \mathcal{Z}_{q}$ which are trivialized by a base change $B' \to B$ of degree at most $d$ and with at most $t+d(T+T^{+})$ branch points is parametrized by a bounded family.  We denote by $\widetilde{\mathfrak{F}}\to \widetilde{\mathfrak{S}}$ the bounded subfamily of $\mathfrak{F}' \to \mathfrak{S}'$ parametrizing maps $g_{s}: \mathfrak{F}_{s} \to \mathcal{Z}_{s}$ satisfying these properties. 
After taking smooth resolutions and stratifying the base, we obtain $\widehat{\mathfrak F} \to \widehat{\mathfrak S}$ such that each fiber is a good fibration over $B$.
We then shrink $\widehat{\mathfrak{S}}$ by removing all irreducible components $\mathfrak{S}_{j}$ such that the corresponding family $\widehat{\mathfrak{F}}_{j}$ fails to dominate $X$ and we enlarge $V$ by taking the union with the closures of the images of these families.

\textbf{Step 5:}  Finally the verification of the desired properties of $\widehat{\mathfrak{F}} \to \widehat{\mathfrak{S}}$ follows from Step 5 of \cite[Theorem 8.7]{LRT22} without any modification.
 \end{proof}

Now, we prove Theorem~\ref{theorem:non-dominant} using Theorem \ref{theo:positivebounded}:.

\begin{proof}[Proof of Theorem~\ref{theorem:non-dominant}]
It follows from \cite[Theorem 0.2]{KMM92} that there exists a constant $b'$ only depending on $\dim X$ such that $-b'K_X$ is base point free. Let $L= -K_X$. We apply Theorem~\ref{theo:positivebounded} with $\beta = 0, b = (2g(B) + 1)b', T = 0, a = 1$ to obtain $\xi^\dagger$, $V_1$, and a bounded family of good fibrations $\widehat{\mathfrak F} \to \widehat{\mathfrak S}$. Let $V_2 \subset X$ be the closure of the loci swept out by non-dominant families of curves $s : B \to X$ with $\deg (-s^*K_X) \leq \xi^\dagger$.  Note that since the parameter space for such curves has finite type, $V_{2}$ is a proper closed set of $X$.

We claim that $V = V_1 \cup V_2$ satisfies our assertion. Indeed, suppose that we have a non-dominant family $M \subset \mathrm{Mor}(B, X)$ of anticanonical degree $\geq \xi^\dagger$. Suppose that a general curve parametrized by $M$ is not contained in $V$. By the universal property, a general $C$ comes from a relatively free section $C'$ in a member of our bounded family of varieties $\widehat{\mathfrak F} \to \widehat{\mathfrak S}$. Then it follows from \cite[Lemma 8.5]{LRT22} that such $C'$ deforms to other varieties in our family so that it dominates an entire irreducible component of $\widehat{\mathfrak F}$. However, such a component maps dominantly to $X$. This is a contradiction. Thus every non-dominant family parametrizes curves in $V$.
\end{proof}

\section{Distinctions between relative and absolute case}

There are several ways in which the absolute case is fundamentally different from the relative case.  We briefly explain the key distinctions.

The first difference is the behavior of adjoint rigidity.  Suppose that $\pi: \mathcal{X} \to B$ is a Fano fibration and that $f: \mathcal{Y} \to \mathcal{X}$ is a generically finite map such that $\mathcal{Y}$ admits a family of sections of high degree that is ``large'' in the corresponding component on $\mathcal{X}$.   \cite[Theorem 7.10]{LRT22} shows that if these sections go through sufficiently many general points on $\mathcal{Y}$ then $(\mathcal{Y}_{\eta},-f^{*}K_{\mathcal{X}/B})$ is adjoint rigid.

The analogous statement in the absolute setting for maps $f: Y \to X$ is no longer true, as demonstrated by the following example.

\begin{example} \label{exam:notadjointrigid}
Let $X$ be a general quartic threefold.  For sufficiently positive $d>0$ we will construct a curve $B$ of genus $801$ and a component $N_{d}$ of $\Mor(B,X)$ such that the curves parametrized by $N_{d}$ sweep out a surface $S$ with the following properties:
\begin{enumerate}
\item $(S,-K_{X}|_{S})$ is not adjoint rigid.
\item The curves parametrized by $N_{d}$ have anticanonical degree $d$ and go through at least $\max\{0, 2d-1921\}$ general points of $S$
\end{enumerate}
Thus there is no bound we can impose on the number of general points contained in our curves which will force $S$ to be adjoint rigid.

Let $S$ denote the surface in $X$ swept out by lines.  By \cite{Tennison74} $S$ is contained in $|80H|$.  Furthermore, if we denote the universal family of lines by $g: Y \to B$ with evaluation map $f: Y \to X$ then $g$ is a $\mathbb{P}^{1}$-bundle over a smooth curve $B$ of genus $801$ and $f: Y \to S$ is birational.  

Let $\mathcal{E}$ denote the rank $2$ locally free sheaf on $B$ such that $Y = \mathbb{P}_{B}(\mathcal{E})$.  Then $K_{Y/B} = g^{*}c_{1}(\mathcal{E}) - 2\xi$ where $\xi$ is a divisor representing $\mathcal{O}_{Y/B}(1)$.  On the other hand, if we write $-f^{*}K_{X} \equiv aF + b\xi$ where $F$ denotes a fiber of $g$ then we have
\begin{equation*}
-f^{*}K_{X} \cdot F = 1 \implies b = 1 \qquad \qquad (-f^{*}K_{X})^{2} = 320 \implies b^{2}c_{1}(\mathcal{E}) + 2ab = 320
\end{equation*}
so that $-f^{*}K_{X} \sim_{num} \xi + 160F - \frac{1}{2}g^{*}c_{1}(\mathcal{E})$.  
Thus $K_{Y/B} - 2f^{*}K_{X} \sim_{num} 320F$.  In particular, since $f: Y \to S$ is birational we conclude that $(S,-K_{X}|_{S})$ is not adjoint rigid.

When $d$ is sufficiently large there is an irreducible component $N_{d} \subset \Sec(Y/B)$ parametrizing sections of $g$ satisfying $-f^{*}K_{X} \cdot C = d$. Then we have
\begin{align*}
\dim(N_{d}) & \geq -K_{Y/B} \cdot C + (1-g(B)) \\
& = (2d - 320) + (1-g(B))
\end{align*}
Let us show that $N_{d}$ is also an irreducible component of $\Mor(B,X)$.  If the curves deformed out of $Y$ then they would lie in an irreducible component $M$ parametrizing a dominant family.  Then we would have
\begin{align*}
\dim(M) & \leq d + 3(1-g(B)) + h^{1}(B,s^{*}T_{X}) \\
& \leq d + 3(1-g(B)) + h^{1}(B,T_{B}) + h^{1}(B,s^{*}T_{X}/T_{B}) \\
& \leq d + 3(1-g(B)) + (3g(B)) + (2g(B))
\end{align*}
where the final inequality follows from Lemma \ref{lemm:ggh1bound}.  Since this is less than $\dim(N_{d})$, we conclude that when $d$ is sufficiently large $N_{d}$ is an irreducible component of $\Mor(B,X)$.  By \cite[Lemma 3.6]{LRT22} the sections parametrized by $N_{d}$ go through at least $(2d-320) - 2g(B) + 1$ general points of $Y$.
\end{example}

Rather, the correct statement is the following.

\begin{theorem} \label{theo:absadjrigid}
Let $X$ be a smooth projective Fano variety and let $B$ be a smooth projective curve.  Fix a positive integer $T$.  There is some constant $\Gamma = \Gamma(\dim(X), g(B),T)$ with the following property.

Suppose that $f: Y \to X$ is a morphism that is generically finite onto its image and $N$ is an irreducible component of $\Mor(B,Y)$ parametrizing a dominant family of curves $C$ on $Y$ such that
\begin{equation*}
\dim(N) \geq -K_{X} \cdot C + \dim(X)(1-g(B)) - T.
\end{equation*}
Suppose that $a(Y,-f^{*}K_{X}) = 1$.  Then either:
\begin{enumerate}
\item $(Y,-f^{*}K_{X})$ is adjoint rigid, or
\item deformations of the corresponding sections on $Y \times B$ go through at most $\Gamma$ general points of $Y \times B$.
\end{enumerate}
\end{theorem}

\begin{proof}
This follows immediately from \cite[Theorem 7.10]{LRT22} applied with $a_{rel} = 1$, $\beta = 0$, $E=0$, and $b$ a positive integer only depending on $\dim(X)$ chosen so that $|-bK_{X}|$ is very ample.
\end{proof}

A second difference between the relative and absolute settings is the formulation of boundedness statements. Loosely speaking, \cite[Theorem 8.8]{LRT22} shows that in the relative setting all non-free curves can be accounted for by the union of a closed set and twists of a finite set of dominant morphisms.  It is natural to wonder if in the absolute setting we can ``remove the twists'': does a Fano variety $X$ admit a closed set and a finite collection of generically finite maps $f: Y \to X$ which account for all non-free curves?  The following example answers this question in the negative.

\begin{example} \label{exam:needtwists}
Let $X$ be the Fano threefold $\mathbb{P}_{\mathbb{P}^{2}}(\mathcal{O} \oplus \mathcal{O}(1))$ equipped with the projective bundle map $g: X \to \mathbb{P}^{2}$.  We will let $H$ denote a divisor representing $g^{*}\mathcal{O}(1)$ and let $E$ denote the rigid section of $g$.  By \cite[Lemma 5.2, Theorem 5.3, Theorem 5.5]{BLRT20} every dominant map $f: Y \to X$ satisfying $a(Y,-f^{*}K_{X}) = a(X,-K_{X})$ will be birationally equivalent to a projection map $\widehat{f}: S \times_{\mathbb{P}^{2}} X \to X$ induced by a generically finite morphism $\psi: S \to \mathbb{P}^{2}$.

Let $B$ be a general hyperelliptic genus $8$ curve so that $B$ admits a unique degree $2$ morphism $h: B \to \mathbb{P}^{1}$ up to automorphisms of $\mathbb{P}^{1}$. We let $\mathcal{L}$ denote the degree $2$ line bundle defining these morphisms. 
Let $M_{d}$ denote the closure of the sublocus of $\Mor(B,X)$ parametrizing maps $s: B \to X$ such that $s$ is birational onto its image, $g \circ s$ is a $2:1$ map onto a conic in $\mathbb{P}^{2}$, and $\deg(-s^{*}K_{X}) = d$. Note that $M_{d}$ is non-empty when $d$ is sufficiently large.  Indeed, the product $G = B \times_{\mathbb{P}^{2}} X$ is a $\mathbb{P}^{1}$-bundle and so admits sections of large degree and we can simply take the image of these sections in $X$. For $d$ large enough, we see that the normal bundle $N_{s/G}$ will have very large degree, so that $N_{s/G}(-p)$ will be globally generated for any $p$, and hence the curves in $M_d$ which pass through $p$ will not all pass through some other point $q$ of $G$. It follows that a general $s$ is birational onto its image in $X$. 

We claim that when $d$ is sufficiently large then $M_{d}$ is an irreducible component of $\Mor(B,X)$ which parametrizes a dominant family of non-free curves.  Suppose the maps parametrized by $M_{d}$ were not dense in an irreducible component.  Then when we compose the general map in this component with $g$ we would obtain a birational morphism $B \to \mathbb{P}^{2}$ onto a degree $4$ curve. But this is not possible since the genus of $B$ is too large.  To show that the general map $s$ parametrized by $M_{d}$ is not free, note that the surjection $T_{X} \to g^{*}T_{\mathbb{P}^{2}}$ yields a surjection $s^{*}T_{X} \to h^{*}(\mathcal{O}(3) \oplus \mathcal{O}(3))$.  Thus
\begin{equation*}
h^{1}(B,s^{*}T_{X}) \geq h^{1}(B,\mathcal{L}^{\otimes 3}) \geq (g(B)-1)-3\deg(\mathcal{L}) > 0
\end{equation*}
and so $s$ is not free.

Next we show that when $d$ is sufficiently large there is no dominant generically finite map $f: Y \to X$ of degree $\geq 2$ and no irreducible component $N \subset \Mor(B,Y)$ such that $f_{*}: N \to M_{d}$ is dominant.  If there were such a map, then Theorem \ref{theo:ainvariantandsectionsabsolutecase} implies that $a(Y,-f^{*}K_{X}) = a(X,-K_{X})$.  Thus $Y$ is birationally equivalent to the projection $S \times_{\mathbb{P}^{2}} X \to X$ for some generically finite map $\psi: S \to \mathbb{P}^{2}$ of degree $\geq 2$.  After replacing $Y$ by a birational model and $N$ by a family of strict transforms, we may assume that $Y$ admits a morphism to $S$.

Consider the images on $S$ of the curves parametrized by $N$ on $Y$.  There are two cases:
\begin{enumerate}
\item The image on $S$ of the general curve parametrized by $N$ is rational.  Then $S$ carries a family $R$ of rational curves $C$ such that $\psi|_{C}$ is an isomorphism 

and $\psi_{*}(R)$ is dense in the family of conics on $\mathbb{P}^{2}$.  Since the preimage of a general conic is irreducible by the Bertini theorem, the only possibility is that $\psi: S \to \mathbb{P}^{2}$ is also degree $1$ and hence birational.  This contradicts our assumption that $\deg(f) \geq 2$.

\item The image on $S$ of the general curve parametrized by $N$ is birational to $B$.  Then $S$ carries a family $R$ of curves $C$ birational to $B$ such that $\psi|_{C}$ is a $2:1$ cover of a conic in $\mathbb{P}^{2}$ and $\psi_{*}(R)$ is dense in the family of conics on $\mathbb{P}^{2}$.  Note that the $\psi$-preimage of a general conic is smooth by the Bertini theorem so that a general curve parametrized by $R$ is isomorphic to $B$.

Let $D$ denote the branch divisor of $\psi$.  Let $U$ denote the parameter space of conics in $\mathbb{P}^{2}$.  As we vary  $Q \in U$ the intersection $D \cap Q$ defines a morphism $U \to \overline{M}_{0,2\deg(D)}$ where the latter space parametrizes stable genus $0$ curves with $2 \deg(D)$ marked points.  We claim that the image of this map has dimension at least $1$; indeed, some conics are tangent to $D$ while others will meet $D$ transversally at distinct points.  However, the ramification divisor for the $2:1$ cover $B \to \mathbb{P}^{1}$ corresponds to a unique point in $\overline{M}_{0,2 \deg(D)}$ (since any two such maps are related by an automorphism of $\mathbb{P}^{1}$).  Thus it is impossible for the $\psi$-preimage of every general conic to be isomorphic to $B$.  
\end{enumerate}
Together (1) and (2) show the impossibility of a morphism $f: Y \to X$ and an irreducible component $N \subset \Mor(B,Y)$ as above.  It follows that there is not a finite set of dominant generically finite maps $\{ f_{i}: Y_{i} \to X\}$ of degree $\geq 2$ such that a general curve parametrized by $M_{d}$ can be obtained by composing $f_{i}$ with a map $s: B \to Y_{i}$.
\end{example}

\bibliographystyle{alpha}
\bibliography{absolutecase}

\end{document}